\documentclass[a4paper, 12pt,oneside,reqno]{amsart}
\usepackage[a4paper]{geometry}
\usepackage[english]{babel}
\usepackage[utf8x]{inputenc}
\usepackage[T1]{fontenc}
\usepackage{amsfonts}
\usepackage{mathrsfs}

\usepackage[sc]{mathpazo}
\usepackage{tikz}
\usetikzlibrary{arrows,shapes,snakes,automata,backgrounds,petri,through,positioning}
\usetikzlibrary{intersections}

\usepackage{tikz-cd}

\usepackage[matrix,arrow]{xy}
\usepackage{amssymb,amscd,amsthm,amsmath}

\usepackage{algorithm2e}

\usepackage[a4paper]{geometry}

\usepackage{amsmath}
\usepackage{amssymb}
\usepackage{amsthm}
\usepackage{graphicx}
\usepackage[colorinlistoftodos]{todonotes}
\usepackage[colorlinks=true, allcolors=blue]{hyperref}

\usepackage{BOONDOX-frak} 

\numberwithin{equation}{section}

\newtheorem{theorem}{Theorem}[section]
\newtheorem{lemma}[theorem]{Lemma}

\newtheorem{corollary}[theorem]{Corollary}

\theoremstyle{definition}
\newtheorem{definition}[theorem]{Definition}
\newtheorem{example}[theorem]{Example}
\newtheorem{remark}[theorem]{Remark}

\newtheorem{construction}[theorem]{Construction}

\newcommand{\tarc}{\mbox{\large$\frown$}}
\newcommand{\arc}[1]{\stackrel{\tarc}{#1}}

\usepackage{array}
\newcolumntype{P}[1]{>{\centering\arraybackslash}p{#1}}
\usepackage{multirow}

\title{On meandric permutations}
\author{Viktor Lopatkin}
\address{HSE University, Faculty of Computer Science, Pokrovsky Boulevard 11, Moscow, 109028 Russia \ wickktor@gmail.com
}

\begin{document}


\maketitle
\begin{abstract}
We give criteria for a permutation to be meandric. Furthermore, we construct a bijection between meanders and specific types of Gauss diagrams derived from the Thurston generators of braid groups. This enables us to devise an algorithm for constructing such diagrams and to encode meanders using matrices that are precisely the incidence matrices of the corresponding adjacency graphs of the diagrams. Ultimately, we demonstrate that these matrices are idempotent over the field GF(2) and derive a criterion for a permutation to be meandric.
\end{abstract}

\section*{Introduction}

The term \textit{meander} was introduced by Arnold in \cite{A88} to denote a connected oriented non-self-intersecting curve in the plane intersecting a fixed oriented baseline in $n$ points. The intersections are assumed to be strict crossings (=transversal intersections).  

\begin{figure}[h!]
    \centering
     \begin{tikzpicture}[scale = 0.7]
      \draw[-Latex,line width = 1] (-1,0) to (8,0);
      \draw[line width =1] (2,1) to [out = 200, in = 90] (1,0);
      \draw[line width =1] (1,0) to [out = 270, in = 270] (4,0);
      \draw[line width =1] (4,0) to [out = 90, in = 90] (3,0);
      \draw[line width =1] (3,0) to [out = 270, in = 270] (2,0);
      \draw[line width =1] (2,0) to [out = 90, in = 90] (5,0);
      \draw[line width =1] (5,0) to [out = 270, in = 270] (6,0);
      \draw[line width =1] (6,0) to [out = 90, in = 330] (5,1);
       \fill(8,0) node[above] {$\mathbb{R}^1$};
     \end{tikzpicture}
    \caption{In other words, a meander can also be defined as follows. Take a fixed oriented line $L$ in $\mathbb{R}^2$, a meander of order $n$ (in this case $n = 6)$ is a non-self-intersecting curve in $\mathbb{R}^2$ which transversely intersects the line at $n$ points for some positive integer $n$.}
    \label{meander2}
\end{figure}

Meanders are of interest in physics and computational biology as models of polymer folding \cite{DFGG97}. They arise in other guises in polymer physics, algebraic geometry, and the study of planar algebras, especially the Temperley--Lieb algebra. For applications of meanders, the reader is referred to \cite{A88, DiF95, R83}. There is also a vast mathematical literature devoted to the enumeration of various types of meanders, which has connections to many different subjects, from combinatorics to theoretical physics and more recently to the geometry of moduli spaces \cite{DFGG97}. We refer \cite{Z21} to a brief recent survey of this literature.

The permutation defined by ordering the intersection points, first along the baseline and then along the meander, is called a \textit{meander permutation} or \textit{meandric.}

\begin{figure}[h!]
     \begin{tikzpicture}[scale = 0.7]
      \draw[-Latex,line width = 1] (-1,0) to (8,0);
      \draw[-Latex,line width = 1] (-1,0) to (8,0);
      \draw[line width =1] (2,1) to [out = 200, in = 90] (1,0);
      \draw[line width =1] (1,0) to [out = 270, in = 270] (4,0);
      \draw[line width =1] (4,0) to [out = 90, in = 90] (3,0);
      \draw[line width =1] (3,0) to [out = 270, in = 270] (2,0);
      \draw[line width =1] (2,0) to [out = 90, in = 90] (5,0);
      \draw[line width =1] (5,0) to [out = 270, in = 270] (6,0);
      \draw[line width =1] (6,0) to [out = 90, in = 330] (5,1);
     {\foreach \x in
       { 1,2,3,4,5,6
        }
     {
        \fill(\x,0) node[below] {$\x$};
        \fill(\x,0) circle (4pt) ;
              }
    }
    \fill(8,0) node[above] {$\mathbb{R}^1$};
\end{tikzpicture}
    \caption{This meander corresponds to permutation $ (143256)$}\label{meander4}
\end{figure}

Interest in meandric permutations predated modern interest in the enumerative theory of meanders. They were discussed for instance by P. Rosenstiehl in \cite{R83} as planar permutations. These permutations occur in the analysis of geographical data and have the property that they can be sorted in linear time \cite{R83}. 

A natural question is how to determine whether a given permutation is a meandric permutation?

In this paper, we give an answer to this question (see Corollary \ref{the_main_criteria}) as follows;

\textit{a permutation $\pi$ is meandric if and only if 
    \[
     c_{\pi}(i,j) \equiv \begin{cases}
         1 (\bmod{2}), & \mbox{if there is an inversion $(i,j)$, }\\
         0 (\bmod{2}), & \mbox{if there is no inversion $(i,j)$},
     \end{cases}
    \]
    for any $i,j \in \{1,\ldots, n\}.$
}

Where $c_\pi(i,j)$ are defined as follows. For a given permutation $\pi$ of $n$ numbers $\{1,2,\ldots, n\}$, and for $i,j \in \{1, \ldots, n\}$ we call all inversions of form $(i,k)$, $(j,k)$ \textit{common inversions} for $i,j$, and a number of all common inversions we denote by $c_\pi(i,j)$.

In this paper, we use the following ideas.

\begin{enumerate}
    \item Any permutation $\pi$ gives rise to a Gauss diagram $\mathfrak{G}(\pi)$. We show that $\mathfrak{G}(\pi)$ is realizable if and only if $\pi$ is a meandric permutation (see Theorem \ref{realization=meandric}). 
    \item For any Gauss diagram $\mathfrak{G}$, we correspond a symmetric matrix $M(\mathfrak{G})$ over the field $\mathsf{GF}(2)$ (this matrix is exactly an adjacency matrix for the corresponding chord intersection graph for the Gauss diagram).
    \item Using criteria of realization of Gauss diagrams obtained by B. Shtylla, L. Tradli, and L. Zulli in \cite{STZ09} we deduce (see Theorem \ref{the_main_result}) that a permutation $\pi$ is meandric if and only if
    \[
     (M(\mathfrak{G}(\pi)))^2 \equiv M(\mathfrak{G}(\pi)) \bmod{2}.
    \]
Thus, any meander correspondences to a projector operator.
 \item Next, as a corollary of this result, we deduce the criterion (see Corollary \ref{the_main_criteria}) for a permutation to be meandric.
\end{enumerate}

Finally, we present Algorithm \ref{algorithm} gives a construction of meandric permutations (=meanders).

\subsection*{Acknowledgments:} the author would like to express his deepest gratitude to \textsc{prof. Alexander Zvonkin}, for his great support and for having kindly clarified some very important details, and for his time to read the text. Special thanks are due \textsc{Dr. Yury Belousov} for his useful discussions.

\section{Meanders}

To introduce a strict formal definition of meander, we essentially follow \cite{B22}

\begin{definition}[{\cite[Definition 1,2]{B22}}]\label{meander}
    \textit{A meander $M_n$ of order} $n$ is a triple $(\mathcal{D}, \{m_0, m_1, p_0,p_1\}, \{\mu, \ell\})$ consist of
    \begin{enumerate}
        \item $2$-dimensional disk $\mathcal{D}$;
        \item four distinct points $\{ m_0,m_1, p_0,p_1\}$ on the boundary $\partial \mathcal{D}$ such that there exists a connected component of $\partial \mathcal{D}\setminus \{m_0,p_0\}$ containing $\{m_1,p_1\}$;
        \item $\mu, \ell : [0,1] \to \mathcal{D}$ are smooth proper embeddings such that $\mu(0) = m_0$, $\mu(1) = m_1$, $\ell(0) = p_0$, and $\ell(1) = p_1$, and the curves $\mu$, $\ell$ intersect transversely at $n$ points.
    \end{enumerate}

 Two meanders $M_n = (\mathcal{D}, \{m_0,m_1, p_0, p_1\}, \{\mu, \ell\})$, $M_n' = (\mathcal{D}', \{m_0',m_1', p_0', p_1'\}, \{\mu', \ell'\})$ are \textit{equivalent} if there exists a homeomorphism $f: \mathcal{D} \to \mathcal{D}'$ such that $f(\mu) = \mu'$, $f(\ell) = \ell'$, and $f(m_i) = m_i'$, $f(p_i) = p_i'$ for each $i=1,2$.
\end{definition}

\begin{figure}[h!]
    \centering
    \begin{tikzpicture}[scale =1.7]
        \begin{scope}
            \fill[gray!10] (1.5, 0) circle (1.5);
            \draw[thick] (0, 0) to (3, 0);
            \draw[ultra thick] (0.252781, 0.833333) to[out = 0, in = 90, distance = 6.28318] (0.5, 0)
            to[out = -90, in = -90, distance = 18.8496] (2, 0)
            to[out = 90, in = 90, distance = 6.28318] (1.5, 0)
            to[out = -90, in = -90, distance = 6.28318] (1, 0)
            to[out = 90, in = 90, distance = 18.8496] (2.5, 0)
            to[out = -90, in = 180, distance = 6.28318] (2.74722, -0.833333);
           \draw[help lines] (1.5, 0) circle (1.5);
           \draw[fill] (0.252781, 0.833333) circle (0.05);
           \node at (0.252781, 0.833333) [left] {$m_0$};
           \draw[fill] (2.74722, -0.833333) circle (0.05);
           \node at (2.74722, -0.833333) [right] {$m_1$};
           \draw[fill] (0, 0) circle (0.05);
           \node at (0,0) [left] {$p_0$};
           \draw[fill] (3, 0) circle (0.05);
           \node at (3,0) [right] {$p_1$};
           \fill (0.5,0) circle (0.05);
           \node at (0.5,0) [below left] {$1$};
           \fill (1,0) circle (0.05);
           \node at (1,0) [below left] {$2$};
           \fill (1.5,0) circle (0.05);
           \node at (1.5,0) [below right] {$3$};
           \fill (2,0) circle (0.05);
           \node at (2,0) [below right] {$4$};
           \fill (2.5,0) circle (0.05);
           \node at (2.5,0) [below right] {$5$};
        \end{scope}
       \begin{scope}[xshift = 5 cm]
             \fill[gray!10] (1.5, 0) circle (1.5);
             \draw[thick] (0, 0) to (3, 0);
             \draw[ultra thick] (0.180987, 0.714286) to[out = 0, in = 90, distance = 5.38559] (0.428571, 0)
             to[out = -90, in = -90, distance = 16.1568] (1.71429, 0)
             to[out = 90, in = 90, distance = 5.38559] (1.28571, 0)
             to[out = -90, in = -90, distance = 5.38559] (0.857143, 0)
             to[out = 90, in = 90, distance = 16.1568] (2.14286, 0)
             to[out = -90, in = -90, distance = 5.38559] (2.57143, 0)
             to[out = 90, in = 180, distance = 5.38559] (2.81901, 0.714286);
             \fill (0.428571, 0) circle (0.05);
             \node at (0.428571, 0) [below left] {$1$};

             \fill (0.857143, 0) circle (0.05);
             \node at (0.857143, 0) [below left] {$2$};

             \fill (1.28571, 0) circle (0.05);
             \node at (1.28571, 0) [below right] {$3$};

             \fill (1.71429, 0) circle (0.05);
             \node at (1.71429, 0) [below right] {$4$};

             \fill (2.14286, 0) circle (0.05);
             \node at (2.14286, 0) [below] {$5$};

             \fill (2.57143, 0) circle (0.05);
             \node at (2.57143, 0) [below right] {$6$};
             
             \draw[help lines] (1.5, 0) circle (1.5);
             \draw[fill] (0.180987, 0.714286) circle (0.05);
             \node at (0.180987, 0.714286) [left] {$m_0$};
             \draw[fill] (2.81901, 0.714286) circle (0.05);
             \node at (2.81901, 0.714286)[right] {$m_1$};
             \draw[fill] (0, 0) circle (0.05);
             \node at (0,0) [left] {$p_0$};
             \draw[fill] (3, 0) circle (0.05);
             \node at (3,0)[right] {$p_1$};
       \end{scope}
    \end{tikzpicture}
    \caption{Two meanders.}
    \label{fig:my_label}
\end{figure}

\begin{remark}
    We always draw meanders in such a way that $\mathcal{D}$ is a Euclidean disk in $\mathbb{R}^2$, $\ell$ is a diameter in $\mathcal{D}$, $p_0$ is the left endpoint of $\ell$ and $m_0$ is always drawn above $p_0$, and starting from $p_0$, in the clockwise direction along $\partial \mathcal{D}$, we met $m_0$ and then $m_1$. We denote this arc by $\arc{m_0m_1}$. Hence, we get oriented lines.
\end{remark}

There is a natural way to represent meanders by permutations.

\begin{definition}[{\textbf{Meander Permutation}}] Let $M_n = (\mathcal{D}, \{m_0, m_1, p_0,p_1\}, \{\mu, \ell\})$ be a meander of order $n$. Label all intersection points of $\mu$ and $\ell$ with natural numbers in the order of movement from $p_0$ to $p_1$. Writing down these labels in the order of movement from $m_0$ to $m_1$ along $\mu$, we obtain the \textit{meandric permutation} $\pi(M_n)$. Finally, we say that a permutation $\pi\in \mathfrak{S}_n$ is \textit{meandric} if there exist a meander $M_n$ of order $n$ such that $\pi(M_n) = \pi.$
\end{definition}

This allows us to identify any meander of order $n$ with the corresponding permutation $\pi\in \mathfrak{S}_{n}$ (=symmetric group on a set of size $n$).

From Definition \ref{meander}, it follows that the lines $\mu, \ell$ are homeomorphic. Then their roles can be reversed by imposing an orientation on the $\ell$ instead of the curve $\mu.$

More precisely, for a given meander $M_n = (\mathcal{D}, \{m_0, m_1, p_0,p_1\}, \{\mu, \ell\})$ of order $n$ we consider a homeomorphism $f:\mathbb{R}^2 \to \mathbb{R}^2$ such that $f(\mu) = \ell$, we then get a meander $f(M_n) := (\mathcal{D}, \{m_0, m_1, p_0,p_1\}, \{\ell, \mu\}).$

\begin{lemma}\label{inverse}
    The corresponding meandric permutation of $f(M_n)$ is inverse to the permutation of the meander $M_n$, \textit{i.e.,} $\pi(f(M_n)) = (\pi(M_n))^{-1}.$
\end{lemma}
\begin{proof}
 Moving from $p_0$ to $p_1$ along the curve $\mu$ (instead of $\ell$) and writing down the corresponding labels in order of the movement from $m_0$ to $m_1$ along $\ell$ we obtain the following permutation 
\[
\pi' = \begin{pmatrix}
     \pi(1) & \pi(2) & \ldots & \pi(n) \\ 1&2&\ldots & n
    \end{pmatrix}
\]
which is exactly $(\pi (M_n))^{-1},$ as claimed.    
\end{proof}

For a given permutation $\pi\in \mathfrak{S}_n$ we set $\overline{\pi} \in \mathfrak{S}_{n+1}$, where
\[
 \overline{\pi}: = \begin{pmatrix}
     1 & \ldots & n & n+1 \\
     \pi(1) & \ldots & \pi(n) & n+1
 \end{pmatrix}.
\]

\begin{lemma}\label{n<-->n+1}
 A permutation $\pi \in \mathfrak{S}_n$ is meandric if and only if $\overline{\pi}$ is meandric.   
\end{lemma}

Recall that a \textit{Jordan curve} is (= a \textit{simple closed curve}) in the plane is the image of an injective continuous map of a circle into the plane.

\begin{proof}~\\
(1) Let $\pi$ be meandric and $M_n = (\mathcal{D}, \{m_0,m_1, p_0, p_1\}, \{\mu, \ell\}) \in \mathbf{M}_n$ the corresponding meander. Consider an arc $\arc{m_0m_1}$ of $\partial \mathcal{D}$ such that $p_1 \notin \arc{m_0m_1}$.

Set $\vartheta:= \mu \cup \arc{m_0m_1}$. Then $\vartheta$ is a Jordan curve, then by Jordan curve theorem it divides the plane into two regions, say $I$ and $O$, and we assume that $p_1 \in O.$

Next, let $0 < t_1 < \ldots < t_n <1$ be such that $\mu \pitchfork \ell = \{\mu(t_1), \ldots, \mu(t_n)\}$, here, and further $\pitchfork$ denotes transversal intersection of curves.

On the other hand, $\ell$ divides $\mathcal{D}$ into two parts, say, $A$ and $B$ and assume that $m_1 \in A$, then, by continuously of $\mu$, there is $t_n < t_n'< 1$ such that $\mu(t_n') \in A$.

Take $m_1' \in \partial \mathcal{D} \cap B$ and set $\gamma:[t_n',1] \to O \cap \mathcal{D}$ to be a smooth embedding such that $\gamma(t_n') = \mu(t_n')$, $\gamma(1) \in \partial \mathcal{D}\setminus\{p_0, p_1,m_0, m_1\}$, and $\gamma \cap \ell  = \{q\}$.

Set
\[
 \mu'(t): = \begin{cases}
     \mu(t), & 0 \le t \le t_{n+1}, \\
     \gamma(t), & t_{n+1} \le t \le 1,
 \end{cases} \qquad \mu'(t_{n+1}) = q, \, t_n' < t_{n+1}<1.
\]

It is clear that $\mu'(t)$ is a continuous map $[0,1] \to \mathcal{D}$ such that $\mu|_{[0,t_n']}$ is smooth and $\mu|_{[0,t_n']} \cap \ell = \mu|_{[0,t_n']} \pitchfork \ell = \{m(t_1), \ldots, \mu(t_n)\}$. 

Next, using Whitney Approximation Theorem for maps to $\mu'$, we thus get a smooth proper embedding $\overline{\mu}:[0,1] \to \mathbb{R}^2$ such that $\mu|_{[0,t_n']} \cap \ell = \mu|_{[0,t_n']} \pitchfork \ell = \{m(t_1), \ldots, \mu(t_n)\}$.

Finally, it remains to show that $\mu'(t)$ intersects $\ell$ transversely at $\mu'(t_{n+1})$. To do so, we consider the following map
\[
 F: [0,1] \times \mathrm{B}(r) \to \mathbb{R}^2, \qquad (t, \mathbf{x}) \mapsto \overline{\mu}(t) + \mathbf{x}.
\]
where $B(r)$ is an open ball in $\mathbb{R}^2$ with radius $r>0$. Fixing $t\in [0,1]$, $F(t,\mathbf{x}):=\overline{\mu}(t) + \mathbf{x}$ is just a constant translation of the open ball $\mathrm{B}(r)$. Therefore, even without letting $t$ vary, the differential $\mathrm{D}F_{(t, \mathbf{x})}$ is a surjective map onto $\mathsf{T}_{F(t,\mathbf{x})}\mathbb{R}^2$. Therefore, $F\pitchfork \ell$ and by Thom's Transversality Theorem, it follows that for almost every $\mathbf{x} \in \mathrm{B}(r)$, the map $\overline{\mu}_\mathbf{x}(t):=\overline{\mu}(t)+\mathbf{x}$ is transversal to $\ell$. 

Thus, for $M_n = (\mathcal{D}, \{m_0,m_1, p_0, p_1\}, \{\mu, \ell\}) \in \mathbf{M}_n$ with permutation $\pi:=\pi(M_n)$ we have constructed a meander  
 \[
 \overline{M_n}: =(\mathcal{D}, \{m_0,m_1', p_0, p_1\}, \{\overline{\mu}_\mathbf{x}, \ell\}).
 \]
of order $n+1$, and $\pi(\overline{M}_n) = \overline{\pi}$.

(2) Let $\overline{\pi}$ be meandric and $M_{n+1} = (\mathcal{D}, \{m_0,m_1, p_0, p_1\}, \{\mu, \ell\})$ be the corresponding meander with $\pi(M_{n+1}) = \overline{\pi}$. 

It is clear that $\ell$ divides the disk $\mathcal{D}$ into two parts, say $A$ and $B$, and we put $m_0 \in A$ \textit{i.e.,} $A$ is above $\ell$. As before, we have a Jordan curve $\vartheta: = \mu \cup \arc{m_0m_1}$ which divides the plane into two regions, say $O$ and $I$, and we set $p_1 \in O$.

Let $0 < t_1 < \ldots < t_{n+1} <1$ be such that $\mu \pitchfork \ell = \{\mu(t_1), \ldots, \mu(t_{n+1})\}$. By the form of $\overline{\pi}$, $\mu(t_{n+1})$ is the rightmost intersection point. Take a closed ball $B(\mu(t_{n+1}), r)$ centred at $\mu(t_{n+1})$ and with a small enough radius $r>0$, and $B(\mu(t_{n+1}),r)\cap \mu = \{a,b\}$ and we assume that $a$ is above $\ell$, $b$ is below $\ell.$ Let $b = \mu(t_n')$, where $t_n < t_n' <t_{n+1}$. Consider a smooth proper embedding $\gamma: [t_n',1] \to O \cap B$ where $\gamma(t_n') = \mu(t_n')$, $\gamma(1) \in \partial \mathcal{D}$ and set 
\[
 \mu''(t) : = \begin{cases}
     \mu(t) & 0 \le t \le t_n', \\
     \gamma(t), & t_n' \le t \le 1.
 \end{cases}
\]

Finally, using Whitney Approximation Theorem for maps to $\mu''$, we thus get a smooth proper embedding $\check{\mu}:[0,1] \to \mathbb{R}^2$ which intersect $\ell$ transversally in $n$ points. Thus we get a meander of order $n$
\[
 M_n := (\mathcal{D}, \{m_0,\gamma(1), p_0, p_1\}, \{\check{\mu}, \ell\}),
\]
such that $\pi(M_n) = \pi$. This completes the proof.

\end{proof}

\section{Gauss Diagrams and its Realization}

The main tool of this paper is the Gauss diagram technique. Thus, we start with the corresponding notions and results which frequently will be used.

\begin{definition}
 By a \textit{generic plane curve,} $\gamma: S^1 \to \mathbb{R}^2$ we mean an immersion of an (oriented) circle $S^1$ into a plane $\mathbb{R}^2$ having only transversal double points of self-intersection.
\end{definition}

It is clear that a generic plane curve is determined (up to diffeomorphism) by its sequence in which the double points appear as we go around the curve.

To formalize this concept, we introduce the following definitions.

\begin{definition}[\textbf{Double occurrence word}] A \textit{word} is a sequence of characters (=letters) of an alphabet. A \textit{cyclic word} is a class of words in the quotient space given by the ``being cyclically equivalent'' relation, $u \sim v$,  if the following two conditions hold:
\begin{enumerate}
    \item they have the same length, $|u| = |v|  =n$,
    \item there exists $1\le i \le n$ such that for any $1\le j \le n$, $u_{i+j (\bmod{n})} = v_{j (\bmod{n})}$, here $u_i$ means the $i$th letter in $u$.
\end{enumerate}

A \textit{double occurrence word} is a finite cyclic word in which every letter appearing in the word appears exactly twice.
\end{definition}

\begin{definition}[\textbf{Gauss code}] 
 A double occurrence word $w$ is said to be a \textit{Gauss code}, or a \textit{realizable word} if there exists a generic plane curve $\gamma:S^1 \to \mathbb{R}^2$, with finitely many transversely self-intersections with the following property. There is an assignment of the letters of $w$ to the crossing points of $\gamma$ such that, traversing $\gamma$ in a certain direction, traverses the letters in the cyclic order ad in $w.$
\end{definition}

It is convenient to encode these sequences in the following way:

\begin{construction}[\textbf{Gauss Diagram}]\label{Gauss_daigram}
    For a given generic plane curve $\gamma: S^1 \to \mathbb{R}^2$ with $n$ double points, we move around the circle $S^1$ and mark all the points that are mapped to a double point, and then join each pair of marked points mapped to the same double point by a chord. What we obtain is called a \textit{chord diagram} or \textit{Gauss diagram} of order $n$ of the curve and denoted by $\mathfrak{G}(\gamma)$ (see Fig.\ref{exofgauss} a), b)).
\end{construction}

The two Gauss diagrams whose sets of chords differ only by an orientation-preserving diffeomorphism of $S^1$ are considered equivalent and are not distinguished.

\begin{figure}[h!]
  \begin{center}
\begin{tikzpicture}[scale = 3]
 \draw [line width = 2, name path= a] (0.1,0.8) to [out = 0, in = 180] (0.6,0.2);
 \draw[line width =2, name path= b](0.6,0.2) to [out = 0, in = 270] (1.1,0.6);
 \draw[line width =2, name path= c] (1.1,0.6)to [out = 90, in =0] (0.6, 1);
 \draw[line width =2, name path= d](0.6,1) to [out = 180, in = 90] (0.3,0.4);
 \draw[line width =2, name path= e] (0.3,0.4)to [out= 270, in = 180] (0.7,-0.2);
 \draw[line width =2, name path= f](0.7,-0.2)to [out= 0, in = 270] (1.1,0.2);
 \draw[line width =2, name path= g](1.1,0.2)to [out = 90, in =270] (0.5, 1);
 \draw[line width =2, name path= h] (0.5,1)to [out = 90, in = 180] (0.7, 1.3);
 \draw[line width =2, name path= k](0.7,1.3) to [out= 0, in = 90] (1,0.9);
 \draw[line width =2, name path= l](1,0.9)to [out= 270, in = 60] (0.3, 0);
 \draw[line width =2, name path= m](0.3,0) to [out = 240, in = 180] (0.1,0.8);
 \fill [name intersections={of=a and d, by={1}}]
(1) circle (1pt) node[left] {$1$};
 \fill [name intersections={of=a and l, by={2}}]
(2) circle (1pt) node[above =1mm of 2] {$2$};
 \fill [name intersections={of=b and g, by={3}}]
(3) circle (1pt) node[right = 1mm of 3] {$3$};
 \fill [name intersections={of=c and l, by={4}}]
(4) circle (1pt) node[above right] {$4$};
 \fill [name intersections={of=d and g, by={5}}]
(5) circle (1pt) node[above left] {$5$};
 \fill [name intersections={of=e and l, by={6}}]
(6) circle (1pt) node[left = 1mm of 6] {$6$};
 \fill [name intersections={of=g and l, by={7}}]
(7) circle (1pt) node[above = 1mm of 7] {$7$};
 \begin{scope}[scale = 0.3, xshift = 8cm, yshift = 1.6cm, line width=2]
   \draw (0,0) circle (2);
    \coordinate (1) at (90:2);
    \node at (1) [above =1mm of 1] {$1$};
    \fill (1) circle(2pt);
    \coordinate (2) at (115:2);
    \node at (2) [above =1mm of 2] {$2$};
    \fill (2) circle(2pt);
    \coordinate (3) at (140:2);
    \node at (3) [above left = -1mm of 3] {$3$};
    \fill (3) circle(2pt);
    \coordinate (4) at (165:2);
    \node at (4) [left] {$4$};
    \fill (4) circle(2pt);
    \coordinate (5) at (190:2);
    \node at (5) [left] {$5$};
    \fill (5) circle(2pt);
    \coordinate (6) at (215:2);
    \node at (6) [left] {$1$};
    \fill (6) circle(2pt);
    \coordinate (7) at (240:2);
    \node at (7) [below] {$6$};
    \fill (7) circle(2pt);
    \coordinate (8) at (265:2);
    \node at (8) [below] {$3$};
    \fill (8) circle(2pt);
    \coordinate (9) at (290:2);
    \node at (9) [below] {$7$};
    \fill (9) circle(2pt);
    \coordinate (10) at (315:2);
    \node at (10) [below right] {$5$};
    \fill (10) circle(2pt);
    \coordinate (11) at (340:2);
    \node at (11) [below right] {$4$};
    \fill (11) circle(2pt);
    \coordinate (12) at (5:2);
    \node at (12) [right] {$7$};
    \fill (12) circle(2pt);
    \coordinate (13) at (30:2);
    \node at (13) [right] {$2$};
    \fill (13) circle(2pt);
    \coordinate (14) at (60:2);
    \node at (14) [above right] {$6$};
    \fill (14) circle(2pt);
    \draw[line width =1] (1) -- (6);
    \draw[line width =1] (3) -- (8);
    \draw[line width =1] (5) -- (10);
    \draw[line width =1] (4) -- (11);
    \draw[line width =1] (9) -- (12);
    \draw[line width =1] (2) -- (13);
    \draw[line width =1] (7) -- (14);
   \end{scope}
   \begin{scope}[scale = 0.3, xshift = 14cm, yshift = 1.6cm]
            {\foreach \angle/ \label in
       { 90/1, 40/2, 350/3, 300/4, 250/5, 200/6, 140/7
        }
     {
        \fill(\angle:2.5) node{$\label$};
        \fill(\angle:2) circle (3pt) ;
      }
      }
      \draw(90:2) -- (40:2);
      \draw(90:2) -- (300:2);
      \draw(90:2) -- (250:2);
      \draw(90:2) -- (350:2);

      \draw(40:2) -- (200:2);

      \draw(350:2) -- (300:2);
      \draw(350:2) -- (250:2);
      \draw(350:2) -- (200:2);

      \draw(300:2) -- (200:2);
      \draw(300:2) -- (140:2);

      \draw(250:2) -- (200:2);
      \draw(250:2) -- (140:2);   
   \end{scope}
  \fill(0.5,-0.5) node{$a)$};
  \fill(2.5,-0.5) node{$b)$};
  \fill(4.2,-0.5) node{$c)$};
\end{tikzpicture}
\end{center}
\caption{Example of a) a planar curve with Gauss code $\mathsf{12345163754726}$; b) its Gauss diagram and c) its interlacement graph.}\label{exofgauss}
\end{figure}

\begin{definition}
 For any Gauss diagram $\mathfrak{G}(\gamma)$ of a generic plane curve $\gamma$ we associate a graph, which is called \textit{chord-intersection graph} or \textit{interlacement graph}, $\Gamma(\mathfrak{G}(\gamma)) := (V,E)$ where the set of its vertices is exactly the set of all chords of $\mathfrak{G}$, and a pair of two vertices, say, $(v,u)$ is an edge if and only if the chords corresponding to them are intersected (see Fig. \ref{exofgauss} c) for example).  
\end{definition}

This graph is also called a \emph{circle graph} by graph theorists or a \emph{chord interlacement graph} by knot theorists. 

\begin{definition}
Recall that with any graph $\Gamma = (V,E)$ one can associate its \textit{adjacency matrix} $\mathbf{M}(\Gamma)$ with rows and columns labelled by graph vertices, with a $1$ or $0$ in position $(v_i,v_j)$ according to whether $v_i$ and $v_j$ are adjacent or not, and $0$ in all positions of form $(v_i,v_i)$.    
\end{definition}

For a given Gauss diagram $\mathfrak{G}(\gamma)$, we may construct the curve $\gamma$ as follows

\begin{construction}\label{curve_from_diagram}
 Consider a point that traces a curve in $\mathbb{R}^2$ and mark the future double points on it in their order of appearance on the Gauss diagram; whenever a double point whose partner was already marked is to appear, direct the curve to the partner and make a transversal self-intersection there; continue in this way until run out of double points, then take the curve back to the initial point.
\end{construction}

If a Gauss diagram $\mathfrak{G}$ contains a chord $\mathfrak{c}$, then we write $\mathfrak{c}\in\mathfrak{G}$. We denote by $\mathfrak{c}_0$, $\mathfrak{c}_1$ the endpoints of the chord $\mathfrak{c}\in \mathfrak{G}$. We shall also consider every chord $\mathfrak{c} \in \mathfrak{G}$ together with one of two arcs between its endpoints, and the chosen arc is denoted by $\mathfrak{c}_0\mathfrak{c}_1$.

Further, $\mathfrak{c}_\times$ denotes the set of all chords crossing the chord $\mathfrak{c}$ and $\mathfrak{c}_\parallel$ denotes the set of all chords not crossing the chord $\mathfrak{c}$. We put $\mathfrak{c} \not \in \mathfrak{c}_\times$, and $\mathfrak{c} \in \mathfrak{c}_\parallel$.

Let $\mathfrak{G}$ be a Gauss diagram, and let $\mathfrak{a}$ be a chord of $\mathfrak{G}$. \textit{A $C$-contour}, denoted by $C[\mathfrak{a}]$, consists of the chord $\mathfrak{a}$, a chosen arc $\mathfrak{a}_0\mathfrak{a}_1$, and all chords of $\mathfrak{G}$ such that all their endpoints lie on the arc $\mathfrak{a}_0\mathfrak{a}_1$.

Next, let us consider a plane curve $\gamma:S^1 \to \mathbb{R}^2$, and let $\mathfrak{G}$ be its Gauss diagram. Every chord $\mathfrak{c \in G}$ corresponds to a crossing $c$ of $\gamma$. Thus, to every $C$-contour $C[\mathfrak{c}]$, we can associate a closed path $\gamma[c]$ along the curve $\gamma$. We call $\gamma[c]$ \textit{the loop of the curve} $\gamma$.

\begin{lemma}\label{intersections=chords}
    There is a one-to-one correspondence between self-intersection points of $\gamma[c]$ and all chords from $C[\mathfrak{c}]$.
\end{lemma}
\begin{proof}
    It immediately follows from Construction \ref{curve_from_diagram}.
\end{proof}

\begin{definition}
 A Gauss diagram $\mathfrak{G}$ is called \textit{realizable} if there is a generic plane curve $\gamma:S^1 \to \mathbb{R}^2$ such that $\mathfrak{G} = \mathfrak{G}(\gamma)$.
\end{definition}

\begin{remark}
 In the opposite case, they say that \textit{a Gauss diagram is not realizable}. However, it can be considered as a Gauss diagram of a curve on some surface. This leads to the notation of a genus of a Gauss diagram \cite{M} and virtual knots. It allows us to consider Gauss diagrams corresponding to any double occurrence word.  
\end{remark}

The problem concerning which Gauss diagrams can be realized by knots is an old one and has been solved in several ways. For our purposes, we present results obtained by B. Shtylla, L. Traldi, and L. Zulli in \cite{STZ09}.

\begin{theorem}[{\cite[Theorem 2]{STZ09}}]\label{STZ}
 Let $\mathfrak{G}$ be Gauss diagram, $\Gamma(\mathfrak{G})$ its interlacement graph and $\mathbf{M}(\mathfrak{G})$ its adjacency matrix. Then $\mathfrak{G}$ is realizable if and only if there exists a diagonal matrix $\mathbf{D}$ such that $\mathbf{M}+\mathbf{D}$ is idempotent over the field $\mathsf{GF}(2)$.
\end{theorem}

\section{Thurston Configurations}

We are going to correspond a Gauss diagram for any closed meander. To do so, we need some notations from the braid group theory.

We start with an example. Take a permutation, say, $\pi = \begin{pmatrix}
    1 & 2 & 3 & 4 \\3 & 4& 2 & 1
\end{pmatrix}$, it can be visualized as follows (see Fig.\ref{perm}).
\begin{figure}[h!]
    \centering
     \begin{tikzpicture}[scale=0.5]
      {\foreach \x in
         { 1,2,3,4
        }
         {
        \fill(\x,0) node[above] {$\x$};
        \fill(\x,0) circle (2pt) ;
        \fill(\x,-4) node[below] {$\x$};
        \fill(\x,-4) circle (2pt) ;
        }
       }
       \draw (1,0) -- (3,-4);
       \draw (2,0) -- (4,-4);
       \draw (3,0) -- (2,-4);
       \draw (4,0) -- (1,-4);
    \end{tikzpicture}
    \caption{A visualization of the permutation $\pi$, where any intersection point corresponds to an involution of the perumutation.}\label{perm}
\end{figure}

Let us denote by $I_i$ an interval with endpoints $(i, \pi(i))$; where the left endpoint corresponds to the top number of the point on the Figure. We then see that $I_i \cap I_j \ne \varnothing$ if and only if $i<j$ and $\pi(i)>\pi(j)$, and $I_i \cap I_j = \varnothing$ if and only if $i<j$ and $\pi(i) < \pi(j).$ In other words, any inversion of a permutation can be considered as an intersection of the intervals.

\begin{definition}{\cite[9.1]{EpThur}}
For any permutation, $\pi \in \mathfrak{S}_n$ we consider the following set of pairs $R(\pi) \subseteq \{1,\ldots, n\} \times \{1,\ldots, n \}$,
\[
 R(\pi):=\{(i,j)\, : \, i < j, \,  \pi(i) > \pi(j)\}.
\]    
\end{definition}

For instance, for the permutation above, we have $R(\pi) = \{(1,3), (1,4), (2,3), (2,4), (3,4)\}$.  It is also worth to note that each pair of $R(\pi)$ correspondences to the pair of crossing strands, numbering by numbers on the top in Fig.\ref{perm}. A formalization of this idea is one of the aims of this section. 

\begin{lemma}\cite[Lemma 9.1.6]{EpThur}\label{criteria}
A set $R \subseteq \{1,\ldots, n\} \times \{1,\ldots, n\}$ of pairs $(i,j)$, with $i <j$, comes from some permutation $\pi \in \mathfrak{S}_n$ if and only if the following two conditions are satisfied:
\begin{enumerate}
    \item If $(i,j) \in R$ and $(j,k) \in R$, then $(i,k) \in R$.
    \item If $(i,k) \in R$, then $(i,j) \in R$ or $(j,k) \in R$ for every $j$ with $i<j<k$.
\end{enumerate}
\end{lemma}

\begin{remark}\label{useful_for_R}
    Since any subset $R\subseteq \{1,\ldots, n\} \times \{1,\ldots, n\}$ is a binary relation on $\{1,\ldots, n\}$, then for a given $R(\pi)$ we get a graph $\Gamma(R(\pi))$ with adjacency matrix $M(R(\pi))$. It is clear that $M(R(\pi))$ is exactly the matrix representation of the $R(\pi)$. By the construction of $R(\pi)$, $M(R(\pi))$ is a symmetric matrix with zero entries in the main diagonal and all other entries are either $1$ or $0$. This allows us to consider such matrices as elements of the space of matrices over the field $\mathsf{GF(2)}$. 
\end{remark}

\begin{remark}\label{R=braids}
 The set of elements of form $R(\pi)$ is called \textit{Thurston generators} of the braid groups. They are braids with positive crossings, and any two strands cross at most once. The elements of $R(\pi)$ corresponds to crossing of strings, \textit{i.e.,} if $(i,j) \in R(\pi)$ then $i$th and $j$th strands are crossed. But for our purposes, it is enough to interpret these sets as ``shadows'' of braids, \textit{i.e.,} we consider them on the plane, and hence instead of overlapping strands, we consider just crossing them (see an explanation below).
\end{remark}

\begin{definition}(\textbf{Thurston Configurations})\label{shadows}
    For a given permutation $\pi \in \mathfrak{S}_n$ we consider the following configuration of $n$ lines on the plane $\mathbb{R}^2$ labelled with numbers $1,\ldots, n$ (denote them by $\ell_1,\ldots, \ell_n$ respectively); lines $\ell_i$, $\ell_j$ are transversally intersected (\textit{resp.} not intersected) if and only if $(i,j) \in R(\pi)$ (\textit{resp.} $(i,j) \notin R(\pi))$. Such configuration (and its diffeomorphic image) for a given permutation, we call \textit{Thurston configuration of the permutation} $\pi.$ 
\end{definition}

\begin{remark}
    This definition is well-defined because of Lemma \ref{criteria} and \cite[Lemma 9.1.10]{EpThur}.
\end{remark}

\textbf{We consider permutations as the corresponding bijections, thus, the product of two permutations is defined as their composition as functions, so ${\displaystyle \sigma \cdot \pi }$ is the function that maps any element $x$ of the set to $
\sigma (\pi (x))$.}

\begin{lemma}[Thurston's formulas]\label{Thuston_formulas}
For any $\tau,\sigma \in \mathfrak{S}_n$
\[
R(\sigma\tau) = \left( \tau^{-1}R(\sigma)\setminus R(\tau) \right) \cup \left(R(\tau) \setminus \tau^{-1}R(\sigma) \right),
\]
where the image of a pair under a permutation is defined by taking the image of each component and reordering, if necessary, so the smaller number comes first.
\end{lemma}
\begin{proof}

(1) Let $(i,j) \in R(\sigma \tau)$ then $i<j$ and $\sigma \tau (i) > \sigma \tau (j)$.  
 \begin{itemize}
     \item[i)] If $\tau(i) > \tau(j)$ then $(i,j) \in R(\tau)$ and $(\tau(i), \tau(j)) \notin R(\sigma)$, hence $(i,j) \notin \tau^{-1}R(\sigma)$.
     \item[ii)] If $\tau(i) < \tau(j)$ then $(i,j) \notin R(\tau)$ and $(\tau(i), \tau(j) \in R(\sigma)$, hence $(i,j) \in \tau^{-1}R(\sigma)$.
 \end{itemize}

It follows that if $(i,j) \in R(\sigma \tau)$ then $(i,j) \in R(\tau) \cup \tau^{-1} R(\sigma) \setminus  R(\tau) \cap \tau^{-1} R(\sigma) $, \textit{i.e.,} $R(\sigma \tau) \subseteq \left( \tau^{-1}R(\sigma)\setminus R(\tau) \right) \cup \left(R(\tau) \setminus \tau^{-1}R(\sigma) \right)$.

(2) Let $(i,j) \in  \left(R(\tau) \setminus \tau^{-1}R(\sigma) \right) \cup \left( \tau^{-1}R(\sigma) \setminus R(\tau)\right)$.
 \begin{itemize}
     \item[i)] If $(i,j) \in R(\tau) \setminus \tau^{-1}R(\sigma)$ then $\tau(i) > \tau(j)$ and $(i,j) \notin \tau^{-1} R(\sigma)$, hence $\sigma\tau(i)> \sigma\tau(j)$, thus $(i,j) \in R(\sigma \tau)$.
     \item[ii)] If $(i,j) \in \tau^{-1}R(\sigma) \setminus R(\tau)$ then $\tau(i) < \tau(j)$ and $(\sigma \tau(i), \sigma\tau (j)) \in R(\sigma \tau)$, therefore $(i,j) \in R(\sigma \tau).$
 \end{itemize}

It follows that, $\left( \tau^{-1}R(\sigma)\setminus R(\tau) \right) \cup \left(R(\tau) \setminus \tau^{-1}R(\sigma) \right) \subseteq R(\sigma \tau)$ and this completes the proof.
\end{proof}

\begin{corollary}\label{R(pi^{-1})} $R(\pi^{-1})  = \pi R(\pi),$ for any $\pi \in \mathfrak{S}_n.$
\end{corollary}
\begin{proof}
    Let $\sigma = \pi$, $\tau = \pi^{-1}$, then
    \[
    R(\pi \pi^{-1}) = \left( \pi R(\pi)\setminus R(\pi^{-1}) \right) \cup \left(R(\pi^{-1}) \setminus \pi R(\pi) \right)
    \]
   any by $R(\pi \pi^{-1}) = \varnothing$ the statement follows.
\end{proof}

\begin{definition}
    Set 
    \[
\omega = \begin{pmatrix}
    1 & 2 & \ldots & n-1 & n \\
    n & n-1 & \ldots & 2 & 1
\end{pmatrix}.
\]

The element $R(\omega)$ (see \cite[Section 9.1, (9.1.3)]{EpThur}) is called \textit{the Garside element (braid)} and denoted by $\Delta_n$.
\end{definition}

It is easy to see that, $\Delta_n = \{(i,j) \, : \, 1\le i<j \le n\}$ and $\pi \Delta_n = \Delta_n$ for any $\pi \in \mathfrak{S}_n$. Then by Lemma\ref{Thuston_formulas},
\begin{eqnarray*}
    R(\omega \pi) &=& \left( \pi^{-1} \Delta_n \setminus R(\pi) \right) \cup \left( R(\pi) \setminus \pi^{-1}\Delta_n \right) \\
    &=& \left( \Delta_n \setminus R(\pi) \right) \cup \left( R(\pi) \setminus \Delta_n \right) \\
    &=& \left( \Delta_n \setminus R(\pi) \right) \cup \varnothing \\
    &=&\Delta_n \setminus R(\pi).
\end{eqnarray*}

Set
\begin{equation}\label{Thurston}
    \neg R(\pi) := R(\omega \pi) =  \Delta_n \setminus R(\pi),
\end{equation}
for any $\pi \in \mathfrak{S}_n$.

\begin{remark}\label{negR(mu)}
 It is cleat that for a given $\pi \in \mathfrak{S}_n$ the set $\neg R(\pi)$ can be also described as the following set of pair $(i,j)$, $1\le i,j \le n$,
\[
 \neg R(\pi) = \{(i,j)\, | \, i < j, \, \pi(i) < \pi(j)\}.
\]    
\end{remark}

\begin{figure}[h!]
    \centering
     \begin{tikzpicture}[scale=0.5]
      \begin{scope}[,xshift=-8cm]
      {\foreach \x in
         { 1,2,3,4,5,6
        }
         {
        \fill(\x,0) node[above] {$\x$};
        \fill(\x,0) circle (2pt) ;
        \fill(\x,-6) node[below] {$\x$};
        \fill(\x,-6) circle (2pt) ;
        }
       }
       \draw (1,0) to [out = 270, in =90] (1,-6);

       \draw (4,0) to [out = 270, in =90] (2,-6);
       \draw (5,0) to [out = 270, in =90] (5,-6);
       \draw (6,0) to [out = 270, in =90] (6,-6);
       
       \draw[line width=4, white] (3,0) to [out = 270, in =90] (2,-2) to [out = 270, in = 90] (3,-6);
       \draw (3,0) to [out = 270, in =90] (2,-2) to [out = 270, in = 90] (3,-6);
       
      \draw (2,0) to [out = 270, in =90] (4,-6);
      \end{scope}
      \begin{scope}[xshift=0cm]
       {\foreach \x in
         { 1,2,3,4,5,6
        }
         {
        \fill(\x,0) node[above] {$\x$};
        \fill(\x,0) circle (2pt) ;
        \fill(\x,-6) node[below] {$\x$};
        \fill(\x,-6) circle (2pt) ;
        }
       }
       \draw (6,0) to [out = 270, in = 90] (1,-6);

       \draw (5,0) to [out = 270, in =90] (6,-2) to [out = 270, in =90] (2,-6);
        
       \draw (4,0) to [out = 270, in =90] (3,-2) to [out = 270, in =90] (5,-6);

        \draw (3,0) to [out = 270, in =90] (2,-2) to [out=270, in =90] (4,-6);
        
         \draw (2,0) to [out = 270, in =90] (1,-2) to [out = 270, in =90] (3,-6);
        
             \draw (1,0) to [out = 270, in =180] (3,-2) to [out =0, in =90] (6,-6);
        
    \end{scope}
    \begin{scope}[xshift=8cm]
     {\foreach \x in
         { 1,2,3,4,5,6
        }
         {
        \fill(\x,0) node[above] {$\x$};
        \draw(\x,0) circle (2pt) ;
        \fill[white](\x,0) circle (2pt) ;
        
        \fill(\x,-6) node[below] {$\x$};
        \fill(\x,-6) circle (2pt) ;
        }
       }
      \draw (6,0) to [out = 270, in =90] (1,-6);

        \draw (5,0) to [out = 270, in =90] (1,-4.5) to [out = 270, in =90] (2,-6);
      
      \draw (4,0) to [out = 270, in =90] (1,-3.5) to [out = 270, in =90] (3,-6);

      \draw (3,0) to [out = 270, in =90] (1,-2.5) to [out = 270, in =90] (4,-6);

      \draw (2,0) to [out = 270, in =90] (1,-1.5) to [out = 270, in =90] (5,-6);
      
         \draw (1,0) to [out = 270, in =90] (6,-6);
    \end{scope}
     \fill(-5,-8.5) node{$a)$};
     \fill(3.5,-8.5) node{$b)$};
     \fill(12,-8.5) node{$c)$};
  \end{tikzpicture}
  \caption{For the permutation $\pi = \begin{pmatrix}
    1 &2 &3 & 4 & 5 & 6 \\
    1 & 4 & 3 & 2 & 5 & 6 \end{pmatrix}$ we have: a) $R(\pi)$, b) $\neg R(\pi) = R(\omega \pi)$, and c) the Garside element correspondences to the set $\Delta_6.$}\label{braids}
\end{figure}

\begin{example}\label{ex}
Let $\pi =  \begin{pmatrix}
    1 &2 &3 & 4 & 5 & 6 \\
    1 & 4 & 3 & 2 & 5 & 6 \end{pmatrix}$, then we obtain $R(\pi) = \{(2,3), (3,6), (3,4)\}$, $\neg R(\pi) = \{(1,2), (1,3), (1,4), (1,5), (1,6), (2,5), (2,6), (3,5), (3,6), (4,5), (4,6), (5,6)\}$. Next, we have
\[ \omega \pi = \begin{pmatrix}
    1 & 2 & 3 & 4 & 5 &6 \\
    6 & 5 & 4 & 3 & 2 &1
\end{pmatrix} \begin{pmatrix}
    1 &2 &3 & 4 & 5 & 6 \\
    1 & 4 & 3 & 2 & 5 & 6
\end{pmatrix} = \begin{pmatrix}
    1 & 2 & 3 & 4 & 5 & 6\\
    6 & 3 & 4 & 5& 6 &1
\end{pmatrix}
\]

Then, $R(\omega \pi) = \{(1,2), (1,3), (1,4), (1,5), (1,6), (2,5), (2,6), (3,5), (3,6), (4,5), (4,6), (5,6)\}$, i.e., $R(\omega \pi)  = \neg R(\pi).$ The homeomorphic image of the corresponding Thurston configurations are shown in Fig.\ref{braids}.
\end{example}

Denote by $\Omega_n$ a $n\times n$ matrix such that all its entries are $1$, except elements of the main diagonal which are equal to $0$.

\begin{lemma}\label{M+M'=1}
    For any permutation $\pi\in \mathfrak{S}_n$, 
     \[
     M(R(\pi)) + M(R(\omega \pi)) = \Omega_n.
    \] 
\end{lemma}

\begin{proof}
     It is clear that $\Omega_n = M(\Delta_n)$. Finally, by $\neg R(\pi) := R(\omega \pi) =  \Delta_n \setminus R(\pi)$, the statement follows.
\end{proof}

\section{Meanders, its Graphs, and its Gauss Diagrams} \label{sec:meanders}

This is a key section of this paper. We construct a Gauss diagram $\mathfrak{G}$ for any  meandric permutation. Although interesting machinery can be developed for meandric permutations, it is natural (and convenient) to begin with an arbitrary permutation $\pi \in \mathfrak{S}_n$.

\begin{definition}[\textbf{a Gauss diagram and a graph for a permutation}]\label{Gauss_for_meander}
Let $\pi \in \mathfrak{S}_n$ be a permutation on the set $\{1,2,\ldots, n\}$. Set $A = \{0,\theta_n, 1,2,\ldots, n \}$, where $\theta_n$ is assumed to be an empty symbol if and only if $n \equiv 0 (\bmod{2})$. Consider the following double occurrence word $W(\pi)$ in letters of the alphabet $A$, written as a sequence, $(0, 1, 2, \ldots, n, \theta_n,0, \pi(1), \pi(2), \ldots, \pi(n), \theta_n).$ 

The corresponding Gauss diagram is called \textit{Gauss diagram for the permutation}, and denoted by $\mathfrak{G}(\pi)$, and the corresponding chord-intersection graph (= interlacement graph) we simply denote by $\Gamma(\pi)$.

In the case when a permutation $\pi$ is meandric we call $\mathfrak{G}(\pi)$ and $\Gamma(\pi)$ as a \textit{meandric Gauss diagram} and a \textit{meandric graph} respectively.
\end{definition}

\begin{lemma}\label{Gauss_for_permutation}
    For a given permutation $\pi \in \mathfrak{S}_n$, the Gauss diagram $\mathfrak{G}(\pi)$ can be described as follows, 
    \[
\mathfrak{G}(\pi): = \mathcal{D} \cap \{\ell_0, \ldots, \ell_n \} \cup \partial \mathcal{D},
\]
where all $\ell_i$, $i =1, \ldots, n$ form Thurston configuration $R(\omega\pi^{-1})$, $\ell_0, \ell_{\theta_n},\ell_1, \ldots, \ell_{\theta_n}$ are extra lines intersect all $\ell_i$,  $\mathcal{D}\subset \mathbb{R}^2$ is a disk such that $\mathcal{D}\setminus \partial \mathcal{D}$ contains all intersections of $\ell_0,\ell_{\theta_n}, \ell_1,\ldots, \ell_n$, and $\ell_{\theta_n}$ is empty if and only if $n \equiv 0(\bmod{2}).$
\end{lemma}
\begin{proof}
  By Definition \ref{Gauss_for_permutation}, a Gauss diagram $\mathfrak{G}(\pi)$ looks like in Fig.\ref{chord_for_proof}. It follows that that corresponding permutation $\pi'$ has the following form
  \[
 \pi'=  \begin{pmatrix}
      \pi(n) & \pi(n-1) & \ldots & \pi(2) & \pi(1) \\
      1 & 2 & \ldots & n-1 & n
  \end{pmatrix}.
  \]
Thus, for any $1 \le i \le n$, $\pi'(\pi(i)) = \omega(i)$, setting $i= \pi^{-1}(j)$ we get $\pi' = \omega \pi^{-1}$. Finally, by Definition \ref{Gauss_for_meander} and Definition \ref{shadows} the statement follows.
\end{proof}

\begin{corollary}\label{meandric_graph}
 A graph $\Gamma(\pi)$ for a permutation $\pi \in \mathfrak{S}_n$ can be also described as follows; $\Gamma(\pi) = (V,E)$ with a set of vertices $V = \{0, \theta_n,1,2,\ldots, 2n\}$ and the set of edges $E: = \bigcup_{1\le i \le n}\{(0,i)\} \cup \{(\theta_n, i)\} \cup R(\omega \pi^{-1}).$
\end{corollary}

\begin{example}
 Let us consider the following permutation 
 \[
 \pi = \begin{pmatrix}
     1 & 2 & 3 & 4 & 5 & 6 \\
       3 & 2 & 1 & 4 & 5 & 6 
 \end{pmatrix},
 \]
 we have 
 \[
  \omega \pi^{-1} =  \begin{pmatrix}
     1 & 2 & 3 & 4 & 5 & 6 \\
     4 & 5 & 6 & 3 & 2 & 1 
 \end{pmatrix}.
 \]

The corresponding Thurston configuration and the Gauss diagram are shown in Fig.\ref{ex4}. It is easy to see that the permutation $\pi$ is meandric, indeed, the corresponding meander $M_6$ is shown in Fig.\ref{meander10}.
\end{example}

\begin{figure}[h!]
    \centering
     \begin{tikzpicture}[scale = 0.8]
          \begin{scope}[xshift = -5cm, scale = 0.7]
         {\foreach \x in
          {1,2,3,4,5,6}
         {
           \fill(\x,0) node[above] {$\x$};
           \fill(\x,0) circle (2pt) ;
           \fill(\x,-6) node[below] {$\x$};
            \fill(\x,-6) circle (2pt) ;
         }
        }
        \draw (1,0) to [out = 270, in = 90] (4,-6);
        \draw (2,0) to [out = 270, in =90]  (5,-6);
        \draw (3,0) to [out = 270, in =90]  (6,-6);
        \draw (4,0) to [out = 270, in =90]  (1,-2) to [out = 270, in = 90] (3,-6);
        \draw (5,0) to [out = 270, in =90]  (1,-4) to [out = 279, in = 90] (2,-6);
        \draw (6,0) to [out = 270, in =90]  (1,-6);
        \end{scope}
        \begin{scope}[xshift = 5cm, yshift = -1.5cm,scale = 0.7]
         {\foreach \angle/ \label in
       {180/0, 155/1, 130/2,  105/3, 80/4, 55/5, 30/6, 0/0,
       330/3, 305/2, 280/1, 255/4, 230/5, 210/6
        }
     {
        \fill(\angle:3.5) node{$\label$};
        \fill(\angle:3) circle (4pt) ;
      }
    }

    {
 \foreach \angle/ \label in 
{330/6, 305/5, 280/4, 255/3, 230/2, 210/1}
     \fill(\angle:4.5) node{$\boxed{\label}$};
}
    
    \draw[line width = 2] (0,0) circle (3);
    
    \draw (0:3) -- (180:3);
    \draw (155:3) -- (280:3);
    \draw (130:3) -- (305:3);
    \draw (105:3) -- (330:3);
    \draw (80:3) -- (255:3);
    \draw (55:3) -- (230:3);
    \draw (30:3) -- (210:3);
        \end{scope}
     \end{tikzpicture}
    \caption{Thurston's configuration $R(\omega \pi^{-1})$ (at the left) and the Gauss diagram $\mathfrak{G}(\pi)$ (at the right); we see that a configuration of all chords with endpoints $1,\ldots, 6$ is exactly as $R(\omega \pi^{-1} )$; the boxed numbers are ends of the corresponding interval starts from the top number.}\label{ex4}
\end{figure}

\begin{figure}[h!]
    \centering
\begin{tikzpicture}[scale = 1.3]
\fill[gray!10] (1.5, 0) circle (1.5);
\draw[help lines] (1.5, 0) circle (1.5);
\draw[ultra thick] (0, 0) to (3, 0);
\draw[ thick] (0.0625277, 0.428571) to[out = 0, in = 90, distance = 16.1568] (1.28571, 0)
 to[out = -90, in = -90, distance = 5.38559] (0.857143, 0)
 to[out = 90, in = 90, distance = 5.38559] (0.428571, 0)
 to[out = -90, in = -90, distance = 16.1568] (1.71429, 0)
 to[out = 90, in = 90, distance = 5.38559] (2.14286, 0)
 to[out = -90, in = -90, distance = 5.38559] (2.57143, 0)
to[out = 90, in = 180, distance = 5.38559] (2.93747, 0.428571);

\draw[fill] (0.0625277, 0.428571) circle (0.05) node [left] {$m_0$};
\draw[fill] (2.93747, 0.428571) circle (0.05) node [right] {$m_1$};
\draw[fill] (0, 0) circle (0.05) node [left] {$p_0$};
\draw[fill] (3, 0) circle (0.05) node [right] {$p_1$};

\draw[fill] (1.28571, 0) circle (0.025) node [above right] {$3$};
\draw[fill] (0.857143, 0) circle (0.025) node [above right] {$2$};
\draw[fill] (0.428571, 0) circle (0.025) node [above left] {$1$};
\draw[fill] (1.71429, 0) circle (0.025) node [below] {$4$};
\draw[fill] (2.14286, 0) circle (0.025) node [below] {$5$};
\draw[fill] (2.57143, 0) circle (0.025) node [below] {$6$};
\end{tikzpicture}   
    \caption{Meander $M_6$, $\pi(M_6) =  \begin{pmatrix}
     1 & 2 & 3 & 4 & 5 & 6 \\
       3 & 2 & 1 & 4 & 5 & 6 
 \end{pmatrix} = \pi$.}
    \label{meander10}
\end{figure}

It is worth noting that the Gauss diagram $\mathfrak{G}(\pi)$ is realizable, indeed, we get the following curve (see Fig.\ref{curve_with_meander}). We see that if we intersect this curve with a disc (in the figure the disk has a dashed boundary) we then get exactly the meander.


\begin{figure}[h!]
      \begin{tikzpicture}
          \begin{scope}[xshift = -4cm]
           \draw[name path =o,line width = 1.5] (0,0) to [out = 60, in = 90] (5,0) to [out=270, in = 270] (0,0);     
           \draw (0,0) to [out = 90, in = 210] (1,2);
           \draw[name path = a] (1,2) to [out = 30, in = 70] (3.2,0.8);
           \draw[name path =b] (3.2,0.8) to [out = 250, in = 300] (1.9,1.5);
      
      \draw[name path =c] (1.9,1.5) to [out = 120, in = 90] (1, 1);
      
      \draw[name path = d] (1,1) to [out =  270, in = 200] (3,0.4);
      
      \draw[name path =e] (3,0.4) to [out = 20, in = 220] (4,1.5);

      \draw[name path =f] (4,1.5) to [out = 40, in = 90] (4.5,0.8);

      \draw[name path =h] (4.5,0.8) to [out = 270, in = 180] (5,0.2);

      \draw (5,0.2) to [out = 0, in = 90] (5.5,-0.5)  to[out = 270, in = 0] (2.7,-1.9) to [out = 180, in = 240] (0,0);
     
     \fill [name intersections={of=o and a, by={A}}]
       (A) circle (2pt) node[above] {$3$};  
     
     \fill [name intersections={of=o and b, by={B}}]
       (B) circle (2pt) node[above right] {$2$};
       
     \fill [name intersections={of=o and d, by={C}}]
       (C) circle (2pt) node[above left] {$1$};
       
     \fill [name intersections={of=o and e, by={D}}]
       (D) circle (2pt) node[above left] {$4$}; 
      
     \fill [name intersections={of=o and f, by={E}}]
       (E) circle (2pt) node[above right] {$5$};
       
     \fill [name intersections={of=o and h, by={F}}]
       (F) circle (2pt) node[above right] {$6$}; 
     
       \draw [fill](0,0) circle (2pt);
       \node[left] at (0,0){$0$};
          \end{scope}
    \begin{scope}[xshift = 4cm]
                    \draw[dashed, name path=s1] (1.5,-0.8) to [out = 330, in = 270] (6,1);
      \draw[dashed, name path=s2] (6,1) to [out = 90, in = 90] (0,1) to [out = 270, in = 150] (1.5,-0.8);
      
      \draw[name path =o,line width = 1.5] (0,0) to [out = 60, in = 90] (5,0) to [out=270, in = 270] (0,0);
     
      \draw[name path = a0] (0,0) to [out = 90, in = 210] (1,2);
      \draw[name path = a] (1,2) to [out = 30, in = 70] (3.2,0.8);
            
      \draw[name path =b] (3.2,0.8) to [out = 250, in = 300] (1.9,1.5);
      
      \draw[name path =c] (1.9,1.5) to [out = 120, in = 90] (1, 1);
      
      \draw[name path = d] (1,1) to [out =  270, in = 200] (3,0.4);
      
      \draw[name path =e] (3,0.4) to [out = 20, in = 220] (4,1.5);

      \draw[name path =f] (4,1.5) to [out = 40, in = 90] (4.5,0.8);

      \draw[name path =h] (4.5,0.8) to [out = 270, in = 180] (5,0.2);

      \draw[name path = k] (5,0.2) to [out = 0, in = 90] (5.5,-0.5)  to[out = 270, in = 0] (2.7,-1.9) to [out = 180, in = 240] (0,0);
     
     \fill [name intersections={of=o and a, by={A}}]
       (A) circle (2pt) node[above] {$3$};  
     
     \fill [name intersections={of=o and b, by={B}}]
       (B) circle (2pt) node[above right] {$2$};
       
     \fill [name intersections={of=o and d, by={C}}]
       (C) circle (2pt) node[above left] {$1$};
       
     \fill [name intersections={of=o and e, by={D}}]
       (D) circle (2pt) node[above left] {$4$}; 
      
     \fill [name intersections={of=o and f, by={E}}]
       (E) circle (2pt) node[above right] {$5$};
       
     \fill [name intersections={of=o and h, by={F}}]
       (F) circle (2pt) node[above right] {$6$}; 
     
       \draw [fill](0,0) circle (2pt);
       \node[left] at (0,0){$0$};

    \fill [name intersections={of=o and s1, by={p1}}]
       (p1) circle (2pt) node[below] {$p_1$};

    \fill [name intersections={of=o and s2, by={p2}}]
       (p2) circle (2pt) node[right] {$p_0$};

    \fill [name intersections={of=a0 and s2, by={m0}}]
       (m0) circle (2pt) node[left] {$m_0$};

    \fill [name intersections={of=s1 and k, by={m1}}]
       (m1) circle (2pt) node[right] {$m_1$};
          \end{scope}
      \end{tikzpicture}   
    \caption{The plane curve correspondences to the diagram $\mathfrak{G}(\pi)$ (at the left) and an intersection of this curve with a disc with the dashed boundary (at the right); we thus get meander $M_6$.}
    \label{curve_with_meander}
\end{figure}


\begin{theorem}\label{realization=meandric}
    For a given permutation $\pi \in \mathfrak{S}_n$, $\mathfrak{G}(\pi)$ is realizable if and only if $\pi$ is a meandric permutation.
\end{theorem}

\begin{proof}~\\
    (1) Let $\mathfrak{G}(\pi)$ be a realizable diagram, say, by a plane curve $\gamma$, then by Construction \ref{curve_from_diagram}, $\gamma$ contains a loop, say $\gamma_0$, with origin corresponding to endpoints of chord $\ell_0$ and contains intersection points $1,2,\ldots, n.$ First of all we remark that $\Gamma(\pi)$ does not contain chord $\ell_{\theta_n}$. Indeed, by Jordan curve theorem, $\gamma_0$ divides a plane into two regions, it follows that the number of common intersection points of $\gamma_1$ with $\gamma_0$ must be even, where $\gamma_1:=\gamma \setminus \gamma_0.$ Since $\mathfrak{G}(\pi)$ all other chords intersect $\ell_0$ then the set of all common intersection points of the curves one-to-one correspondences to chords $\ell_1, \ldots, \ell_n$, \textit{i.e.,} $n$ is even.

    Next, set $\gamma_0  = \gamma_{01} \cup \gamma_{1n} \cup \gamma_{n0}$, with $\gamma_{01} \cap \gamma_{1n} = \{1\}$, $\gamma_{01} \cap \gamma_{n0} = \{0\}$, $(\gamma_{01} \cup \gamma_{n0}) \cap \gamma_{1n} = \{1,n\}$, and $\gamma_{1n}$ contains all $1,\ldots, n.$

    Chose points $p_0 \in \gamma_{01}$, and $p_1 \in \gamma_{n0}$ such that $p_0,p_1 \ne 0,1,n$. Take a two-dimensional disk $\mathcal{D}$ such that $\partial \mathcal{D} \cap \gamma = \{p_0,p_1\}$ and $\gamma_{1n} \subset \mathcal{D}$. Next, set $\{m_0, m_1\}:=\gamma_1 \cap \partial \mathcal{D}$. Indeed, by Jordan curve theorem $\partial \mathcal{D}$ divides plane into two regions and by construction $0$, $\gamma \cap \mathcal{D}$ belongs to different regions.

    Finally, it is clear that $(\mathcal{D}, \{m_0,m_1,p_0,p_1\}, \{\gamma_0 \cap \mathcal{D}, \gamma_1 \cap \mathcal{D}\})$ is a meander, and the statement follows.

(2) Let $M = (\mathcal{D}, \{m_0,m_1,p_0,p_1\}, \{\mu, \ell \})$ be a meander with permutation $\pi \in \mathfrak{S}_n$. If $n$ is odd we then consider a meander $M' = (\mathcal{D}, \{m_0,m_1',p_0,p_1\}, \{\mu', \ell \})$, where $\mu'$ intersects $\ell$ at the first $n$ points as $\mu$ do and also intersect $\ell$ at an extra point $n+1$ which is the rightmost marked point on $\ell.$

Thus, we may assume that $n$ even.

Since $\ell$ divides $\partial \mathcal{D}$ into two arcs, say, $\arc{p_0p_1}$ and $\arc{p_1p_0}$, then both $m_0,m_1$ belong only one of them, say $\arc{p_0p_1}$, because $n$ is assumed to be even.

It is cleat the $\partial \mathcal{D}$ is a simple curve (a curve that does not cross itself) then, by Jordan curve theorem, it divides $\mathbb{R}^2$ into two regions, say $A,B$ and we set $\mu, \ell \in A.$

Set $\gamma: [0.1] \to \mathbb{R}^2$ be a smooth proper embedding such that $\gamma(0) = m_1$, $\gamma(1) = p_0$ and $\gamma(x) \in B$ for any $0 < x < 1$.

Let $\alpha: = \arc{m_1m_0}$ be an arc of $\partial \mathcal{D}$ such that $p_0,p_1 \in \alpha$ and set $\mu': = \alpha \cup \gamma \cup \ell $. By Whitney Approximation Theorem for maps, we get a smooth closed curve $\widetilde{\mu}$.

Finally, by construction of $\widetilde{\mu}$, its Gauss diagram $\mathfrak{G}(\widetilde{\mu})$ is exactly the same as $\mathfrak{G}(\pi)$. This completes the proof.    
\end{proof}

Thus, it clarifies now why we take an interest in the Gauss diagram technique.

 \begin{theorem}\label{meander_matrix=idempotent}
  An adjacency matrix of any meander graph is idempotent.
\end{theorem}
\begin{proof}

(1) Let $M$ be the adjacency matrix of a Gauss diagram $\mathfrak{G}(\mu)$. Being $M$ symmetric, we then have $M^2 = (\langle m_i, m_j \rangle)_{1\le i,j \le n}$, over $\mathsf{GF}(2)$, where 
\[
\langle m_i, m_j\rangle: = m_{i,1}m_{j,1} + \cdots + m_{i,n}m_{j,n},
\]
and $m_k:=(m_{k,1}, \ldots, m_{k,n})$ is the $k$th row of the $M$.

(2) Let $D = \mathrm{diag}(\mathbf{X}_1,\ldots, \mathbf{X}_n)$ be a diagonal $n\times n$ matrix over $\mathsf{GF}(2)$. Since $M$ is symmetric then the matrix $M':=DM +MD$ is $M' = (m'_{i,j})_{1 \le i,j \le n}$, $m'_{i,j} = (\mathbf{X}_i + \mathbf{X}_j)m_{i,j}$ for all $1\le i, j \le n.$ We have $(M+D)^2 = M^2 + MD + DM + D^2$, is it clear that $D^2 = D$.

Next, by Theorem \ref{STZ}, the diagram $\mathfrak{G}(\mu)$ is realizable if and only if $(M+D)^2 = M+D$, \textit{i.e.,} $M^2 + M = MD +DM$. Thus we obtain the following system of linear equations
    \[
   \left\{  m_{i,j} \mathbf{X}_i + m_{i,j}\mathbf{X}_j = \langle m_i, m_j \rangle + m_{i,j}, \quad 1\le i,j \le n  \right.
  \]
and hence $\mathfrak{G}(\mu)$ is realizable if and only if this system has a solution over a field $\mathsf{GF}(2).$ 

Remark that if $m_{i,j}=0$ then $\langle m_i, m_j \rangle =0$ and in the case $i=j$ we obtain $m_{i,i} = \langle m_i, m_i \rangle = 0$. It follows that $\mathfrak{G}(\mu)$ is realizable if and only if the system  
 \[
  \mathbf{X}_i + \mathbf{X}_j = \langle m_i, m_j \rangle + 1, \qquad (i,j) \in K
 \]
 has a solution over the field $\mathsf{GF}(2)$, where $K \subseteq \{1,\ldots, n\}\times \{1,\ldots, n\}$ is a subset such that whenever $(i,j) \in K$ then $m_{i,j} = 1$.

(3) Let us consider now the corresponding entrancement graph $\Gamma(\mu)$ and a cycle $C = (c_1,\ldots, c_\ell)$, where we assume that $c_1 = (i_1,i_2), \ldots, c_\ell = (i_\ell, i_1)$, we then get 
 \[
  \begin{cases}
   \mathbf{X}_{i_1} + \mathbf{X}_{i_2} \equiv \omega((i_1,i_2)) + 1, \\
   \phantom{\mathbf{X}_{i_1}\,}\vdots \phantom{+\mathbf{X}_{i_2}} \ddots \phantom{\omega((i_1,i_2))} \vdots \\
   \mathbf{X}_{i_\ell} + \mathbf{X}_{i_1} \equiv \omega((i_\ell, i_1)) + 1,
  \end{cases} 
 \]
therefore $\mathfrak{G}(\mu)$ is realizable if and only if $\langle m_{i_1}, m_{i_2} \rangle + \cdots +\langle m_{i_\ell}, m_{i_1} \rangle \equiv (0 \bmod{2})$ for any cycle $C$ in the $\Gamma(\mu)$.

(4) We have $m_{0,i} = 1$ for any $1 \le i \le n$ where we have set that $M=(m_{i,j})_{0\le i,j \le n}$ is a matrix of a meander graph. Let $m_{i,j} = 1$ for some $1\le i,j \le n$ then the graph contains a cycle with three edges $(v_0,v_i), (v_i,v_j), (v_j,v_0)$ thus, by the discussion above, $\langle m_0,m_i \rangle + \langle m_i, m_j \rangle + \langle m_0, m_j \rangle \equiv 1 \bmod{2}$. Since $\langle m_0, m_k \rangle  =1 + \langle m_k, m_k \rangle \equiv 1 \bmod{2}$ because of $\langle m_k, m_k \rangle \equiv 0 \bmod{2}$ for any $1 \le k \le n$. Therefore $\langle m_i, m_j \rangle \equiv 1 \bmod{2}$. We thus get $\langle m_i, m_j \rangle \equiv m_{i,j} \bmod{2}$ for any $0 \le i,j \le n$, \textit{i.e.,} $M^2 = M$ as claimed
\end{proof}

\begin{corollary}\label{meander_graph_criteria}
 Let $\Gamma = (V,E)$ be a finite graph with $V = \{0,1,\ldots, 2n\}$, $n >0$. The finite graph $\Gamma = (V,E)$ is a meandric graph if and only if the following conditions hold:
 \begin{enumerate}
     \item $(0,i) \in E$ for any $1\le i \le n$
     \item if $(i,j) \in E$ then either $(i,k) \in E$ or $(k,j) \in E$ for all $i < k < j$
     \item if $(i,j)\in E$ and $(j,k) \in E$ then $(i,k) \in E$ for any $0 \le i,j, k \le 2n$
     \item its adjacency matrix $M$ is idempotent.
 \end{enumerate}
\end{corollary}
\begin{proof}
 It immediately follows from the construction of a meandric graph (see Remark \ref{meandric_graph}), Lemma \ref{criteria} and Theorem \ref{meander_matrix=idempotent}. 
\end{proof}

\section{The Main Result}

Recall that $\Omega_n$ denotes a matrix of size $n\times n$ which all its entries are $1$ except elements of the main diagonal which are assumed to be zero.

\begin{lemma}\label{[M,O]=0}
    Let $M$ be a symmetric matrix over the field $\mathsf{GF}(2)$ with zero diagonal and  sum of all elements of any row is $0$, then $[M, \Omega]$ is a zero matrix over $\mathsf{GF}(2).$
\end{lemma}

\begin{proof}
    Let $M_1,\ldots, M_n$ be all rows of $M$, let $1\le i \le j \le n$. We have
    \begin{eqnarray*}
        \langle M_i, \Omega_j \rangle &=& m_{1,i} + \cdots + m_{i-1, i} + m_{i,i+1} + \cdots + \widehat{m_{i,j}} + \cdots + m_{i,n} \\
        &=& m_{i,j},\\
        \langle \Omega_i, M_j \rangle &=& m_{1,j} + \cdots + \widehat{m_{i,j}} + \cdots + m_{j-1,i} + m_{j,j+1} + \cdots + m_{j,n} \\
        &=& m_{i,j}.
    \end{eqnarray*}

    Thus, $\langle M_i + \Omega_i, M_j + \Omega_j = m_{i,j} + m_{i,j} \equiv 0 \bmod{2}$, and the statement follows.
\end{proof}

Now we are in a position to give criteria for permutation to be meandric. As we have mentioned before (see Remark \ref{useful_for_R}), for a given permutation $\pi$ we have a matrix $M(R(\pi))$ of the corresponding Thurston configuration $R(\pi)$. Just for convenience, we give an explicit reformulation of such matrices.

\begin{definition}\label{matrix_for_permutation}
 For a given permutation $\pi \in \mathfrak{S}_n$, we construct a symmetric matrix $M(\pi) = (m_{i,j})_{1\le i,j \le n}$ of the size $n\times n$ with coefficients in the field $\mathsf{GF}(2)$ as follows; 
 \[
  m_{i,j} : = \begin{cases}
   0 & \mbox{if $i=j$,}\\
   1 & \mbox{if $i<j$ and $\pi(i) > \pi(j)$},\\
   0 & \mbox{if $i<j$ and $\pi(i) < \pi(j)$},
  \end{cases}
 \]
 for all $1 \le i<j \le n.$
\end{definition}

\begin{lemma}\label{M<-->M'}
    For a given permutation $\pi \in \mathfrak{S}_n$, 
    \begin{gather*}
        M^2(\pi) \equiv M(\pi) \bmod{2} \Longleftrightarrow M^2(\pi^{-1}) \equiv M(\pi^{-1}) \bmod{2}, \\
        M^2(\pi) \equiv M(\pi) + \Omega_n \bmod{2} \Longleftrightarrow M^2(\pi^{-1}) \equiv M(\pi^{-1})\bmod{2}.
    \end{gather*}
\end{lemma}

\begin{proof}
  Let $M = M(\pi) = (m_{i,j})_{1 \le i,j \le n}$, $M'=M(\pi^{-1}) = (m_{i,j}')_{1\le i,j\le n}$. By Corollary \ref{R(pi^{-1})}, $R(\pi^{-1}) = \pi R(\pi)$, hence $m_{\pi(i), \pi(j)} \in (M(\pi^{-1}))_{\pi(i)}$ for any $1 \le i,j \le n$, where $A_{(i)}$ means the $i$-th string of a matrix $A.$ it follows that $m'_{i,j} = m_{\pi^{-1}(i), \pi^{-1}(j)}$ and then 
  \[
   \langle M'_{(p)}, M'_{(q)} \rangle = \sum_{j=1}^n m_{\pi^{-1}(p), \pi^{-1}(j)} m_{\pi^{-1}(q), \pi^{-1}(j)}  = \langle M_{(p)}, M_{(q)} \rangle,
  \]
  and the statement follows.
 \end{proof}

\begin{theorem}\label{the_main_result}
 A permutation $\pi \in \mathfrak{S}_{n}$ is meandric if and only if $(M(\pi))^2 \equiv M(\pi)  (\bmod{2})$.
\end{theorem}
\begin{proof}~

(1) Let $n$ be even. 

(a) Let $(M(\pi))^2 = M(\pi)$. By Lemma \ref{M+M'=1}, $M(\omega \pi^{-1}) = \Omega_n + M(\pi^{-1})$. We have
\[
 (M(\omega \pi^{-1}))^2 = \Omega_n^2 + [\Omega_n, M(\pi^{-1})] + (M(\pi^{-1}))^2
\]

By Lemma \ref{[M,O]=0}, $(M(\omega \pi^{-1}))^2 \equiv \Omega_n^2 + (M(\pi))^2$. It is clear that $\Omega_n^2 \equiv E_n$ -- identity matrix because of $n$ is even. By assumption $(M(\pi))^2 \equiv M(\pi)$, hence, by Lemma \ref{M<-->M'}, 
\[
 (M(\omega \pi^{-1}))^2 \equiv E_n + M(\pi^{-1}).
\]

Let us consider the corresponding graph $\Gamma(\pi) = (V,E)$ (see Definition \ref{Gauss_for_meander} and Corollary \ref{meandric_graph}) for the permutation $\pi$. We have 
\begin{align*}
    & V = \{0,1,2,\ldots, n\},\\
    & E: = \bigcup_{1\le i \le n}\{(0,i)\} \cup R(\omega\pi^{-1})
\end{align*}

Thus, its adjacency matrix $M(\Gamma(\pi))$ has the following form
\[
  M(\Gamma(\pi)) = \begin{pmatrix}
      0 & \begin{matrix}
          1 & \ldots & 1 
      \end{matrix} \\
      \begin{matrix}
          1 \\
          \vdots \\
          1
      \end{matrix} & M(\omega\pi^{-1})
  \end{pmatrix}.
\]

Indeed, if we consider the set $R(\omega\pi^{-1})$ as a binary relation then, by Remarks \ref{negR(mu)}, \ref{useful_for_R} and Definition \ref{matrix_for_permutation}, the matrix representation of $R(\omega\pi^{-1})$ is exactly the matrix $M(\omega \pi^{-1})$.

Set $M(\Gamma(\pi)) = (m_{i,j})_{0 \le i,j \le n}$ then we get; $m_{0,0} = 0$, $m_{0,1}= m_{1,0} = \cdots m_{0,n} = m_{n,0}= 1$, and $M(\omega\pi^{-1}) = (m_{i,j})_{1\le i,j \le n}$.

Further, let $(M(\Gamma(\pi)))^2 =  (b_{p,q})_{0\le p,q \le n}$. Being the $M(\Gamma(\pi))$ symmetric, we obtain
\begin{eqnarray*}
    b_{0,0} &=& n \equiv 0 (\bmod{2}),\\
    b_{0,i} &=& m_{1,i} + \cdots + m_{i-1,i} + m_{i,i+1} + \cdots + m_{i,n}\\
    &\equiv& 1 (\bmod{2}),\\
    b_{i,j} &\equiv& \langle m_i, m_j \rangle+1 (\bmod{2}),
\end{eqnarray*}
where, as before, $\langle m_i, m_j\rangle: = m_{i,1}m_{j,1} + \cdots + m_{i,n}m_{j,n},$ and $m_k:=(m_{k,1}, \ldots, m_{k,n})$ is the $k$th row.

Thus, we get
\[ (M(\Gamma(\pi)))^2 = 
 \begin{pmatrix}
     0 & \begin{matrix}
         1 & \cdots & 1
     \end{matrix} \\
     \begin{matrix}
         1 \\
         $\vdots$\\
         1
     \end{matrix} & (M(\omega\pi^{-1}))^2 + J_n
 \end{pmatrix}
 \]
where $J_n$ is a $n\times n$ matrix where every entry is equal to one. We thus have
\begin{eqnarray*}
    (M(\omega\pi^{-1}))^2 + J_n &=& E_n+ M(\pi^{-1}) + J_n \\
    &=& \Omega_n + M(\pi^{-1}) \\
    &=& M(\omega \pi^{-1}).
\end{eqnarray*}

It follows that the matrix $M(\Gamma(\pi))$ is idempotent in the field $\mathsf{GF}(2)$.

Finally, by the construction of $\Gamma(\pi)$ and Lemma \ref{criteria}, $\Gamma(\pi)$ satisfies the first three conditions of Corollary \ref{meander_graph_criteria}.  Hence, by Theorem \ref{meander_matrix=idempotent}, $\Gamma(\pi)$ is a meandric graph for the $\pi$. 
 
(b) Let $\pi$ be a meander, as we have mentioned above the corresponding matrix $M(\Gamma(\pi))$ for the graph $\Gamma(\pi)$ has the following form 
 \[ M(\Gamma(\pi)) = 
 \begin{pmatrix}
     0 & \begin{matrix}
         1 & \cdots & 1
     \end{matrix} \\
     \begin{matrix}
         1 \\
         $\vdots$\\
         1
     \end{matrix} & M(\omega\pi^{-1})
 \end{pmatrix}
 \]
 by Lemma \ref{[M,O]=0},
   \[ (M(\Gamma(\pi)))^2 = 
 \begin{pmatrix}
     0 & \begin{matrix}
         1 & \cdots & 1
     \end{matrix} \\
     \begin{matrix}
         1 \\
         $\vdots$\\
         1
     \end{matrix} & (M(\omega \pi^{-1}))^2 + J_n
 \end{pmatrix}
 \]

 Next, $M(\omega \pi^{-1}) = \Omega_n + M(\pi^{-1})$, then $(M(\omega \pi^{-1}))^2 = E_n + (M(\pi^{-1}))^2$. We then have 
 \[
  (M(\omega \pi^{-1}))^2 + J_n = \Omega_n + (M(\pi^{-1}))^2.
 \]
 
By Corollary \ref{meander_graph_criteria}, $(M(\Gamma(\pi))^2 \equiv M(\Gamma(\pi))$, hence
\[
 \Omega_n + (M(\pi^{-1}))^2 \equiv M(\omega \pi^{-1}) = \Omega_n + M(\pi^{-1}),
\]
and by Lemma \ref{M<-->M'}, the statement follows.

(2) Let $n$ be odd. Then, for a given $\pi \in \mathfrak{S}_n$ we consider 
\[
 \overline{\pi}: = \begin{pmatrix}
     1 & \ldots & n & n+1 \\
     \pi(1) & \ldots & \pi(n) & n+1
 \end{pmatrix}.
\]

By Lemma \ref{n<-->n+1}, $\pi$ is meandric if and only if $\overline{\pi}$ is meandric. Since $n+1$ is even then, we may apply the previous discussion to $\overline{\pi}.$ 

It is clear that

\[
 M(\overline{\pi}) = \begin{pmatrix}
     M(\pi) & \begin{matrix}
         0 \\ \vdots \\ 
     \end{matrix} \\
     \begin{matrix}
         0 & \ldots  
     \end{matrix} & 0
 \end{pmatrix}, \qquad  M^2(\overline{\pi}) = \begin{pmatrix}
     M(\pi)^2 & \begin{matrix}
         0 \\ \vdots \\ 
     \end{matrix} \\
     \begin{matrix}
         0 & \ldots  
     \end{matrix} & 0
 \end{pmatrix}
\]
and the statement follows.
 \end{proof}

For a given permutation $\pi \in \mathfrak{S}_n$, and for $i,j \in \{1, \ldots,n\}$ we call all inversions of form $(i,k)$, $(j,k)$ \textit{common inversions} for $i,j$, and a number of all common inversions we denote by $c_\pi(i,j)$.

Hence, the following result immediately follows from the previous Theorem

\begin{corollary}[{Criteria of meandric permutation}]\label{the_main_criteria}
    A permutation $\pi$ is meandric if and only if 
    \[
     c_{\pi}(i,j) \equiv \begin{cases}
         1 (\bmod{2}), & \mbox{if there is an inversion $(i,j)$, }\\
         0 (\bmod{2}), & \mbox{if there is no inversion $(i,j)$},
     \end{cases}
    \]
    for any $i,j \in \{1,\ldots, n\}.$
\end{corollary}

\section{A Construction of Meandric Permutations}

In this section, we present an algorithm to construct meandric graphs. For a given graph $\Gamma = (V,E)$ we denote by $N(v)$ a set of all neighbors of the vertex $v \in V$ and by $|X|$ we denote a cardinality of the set $X$. 

\begin{lemma}\label{even}
   Let $\Gamma(\pi) = (V,E)$ be a meandric graph then
   \begin{align*}
       & |N(v)| \equiv 0 (\bmod{2}), \\
       & |N(v) \cap N(u)| \equiv 0(\bmod{2}),\\
       & |N(v) \cap N(w)| \equiv 1(\bmod{2})
   \end{align*}
where $v,u,w \in V$, $u \notin N(v)$ and $w \in N(v)$.
\end{lemma}
\begin{proof}
    Indeed, since $M$ is symmetric and $M^2$ (see the proof of Theorem \ref{meander_matrix=idempotent}) we have $M^2 = (m_{i,j}')_{1 \le i,j \le n}$ where $m_{i,j}':= m_{i1}m_{j1} + \cdots + m_{i,n}m_{j,n} = \langle m_i, m_j \rangle$. Hence, by $m_{i,i}=0$, $m_{i,i}'=0$ \textit{i.e.,} $N(v_i) \equiv 0 (\bmod{2})$ for any $1\le i \le n$. Next, if $m_{i,j} = 0$ (\textit{resp.} $m_{i,j}=1$) then the corresponding vertices $v_i, v_j$ are not joined (\textit{resp.} are joined) and by $M^2 = M$, $\langle m_i, m_j \rangle \equiv (0 \bmod{2})$ (\textit{resp.} $\equiv (1 \bmod{2})$). Finally, since $\langle m_i, m_j\rangle = |N(v_i) \cap N(v_j)|$ we complete the proof.
\end{proof}

\begin{remark}
 Since the vertex, $0$ is joined with all other vertices in a meandric graph \textbf{we, just for convenience, omit it.} It follows that by Lemma \ref{even},we have to require the opposite conditions that presented in the Lemma.    
\end{remark}

\RestyleAlgo{ruled}

\begin{algorithm}[hbt!]
\caption{An algorithm to construct meanders}\label{algorithm}
    \KwData{$n \equiv 0 (\bmod{2})$, $V = V_0 \sqcup V_1$, $V_0 =\{2,4, \ldots, N\}$, $V_1 = \{1,3,\ldots, N-1\}$}
\KwResult{$S=\{s_1,\ldots, s_n\}$}
  \While{$V_0,V_1 \ne \varnothing$}{
    \For{$i \leftarrow 0$ \KwTo $1$}{
          choose $v\in V_i$\\
          $T:=\{v\}$\\
          $R:=\cup_{v' >v, v' \notin T}\{(v,v')\}$\\
            \eIf{$\begin{cases} 
                     |N(v)| \equiv 1 (\bmod{2}) \\
                     |N(v) \cap N(u)| \equiv 1 (\bmod{2}) & (u,v) \notin R,\, u \in T \\
                     |N(v) \cap N(w)| \equiv 0 (\bmod{2}) & (w,v) \in R,\, w \in T
                  \end{cases}$\\}
            {$j:=2$\\
            $s_j:=v$\\
            $j:=j+1$\\
            $V:=V\setminus \{v\}$\\
              go back to the beginning}
            {choose another $v\in V_i$}
    }
    }
\end{algorithm}

\begin{example}
  Let $N=8$, then $V_0=\{2,4,6,8\}$ and $V_1=\{1,3,5,7\}$.
  \begin{enumerate}
      \item Picture them as it shown in Fig.\ref{f1} a). Since we have to start with $1$ we thus join it with other vertices. We have to choose a vertex from $V_0$.
  \item \textbf{Choose vertex $6$} and join it with vertices $7,8$ (we draw a circle around a chosen vertex) (see Fig.\ref{f1} b)). It is clear that vertex $6$ has an odd number of neighbors. We have to choose a vertex with an odd number.
  

  \begin{figure}[h!]
      \centering
       \begin{tikzpicture}[scale =0.5]
\begin{scope}[xshift = -10cm]
         {\foreach \angle/ \label in
       { 90/1, 45/2, 0/3, 315/4, 270/5, 225/6, 180/7, 135/8
        }
     {
        \fill(\angle:3.5) node{$\label$};
        \fill(\angle:3) circle (5pt) ;
        \draw[gray] (90:3) -- (\angle:3);
      }
    } 
\end{scope}
       \begin{scope}[xshift = 0cm]
                        {\foreach \angle/ \label in
       { 90/1, 45/2, 0/3, 315/4, 270/5, 225/6, 180/7, 135/8
        }
     {
        \fill(\angle:3.5) node{$\label$};
        \fill(\angle:3) circle (5pt) ;
        \draw[gray] (90:3) -- (\angle:3);
      }
    }
    \draw (225:3) -- (180:3);
    \draw (225:3) -- (135:3);
    
    \draw(225:3) circle(10pt);
       \end{scope}
\begin{scope}[xshift = 10cm]
             {\foreach \angle/ \label in
       { 90/1, 45/2, 0/3, 315/4, 270/5, 225/6, 180/7, 135/8
        }
     {
        \fill(\angle:3.5) node{$\label$};
        \fill(\angle:3) circle (5pt);
        \draw[gray] (90:3) -- (\angle:3);
      }
    }
    \draw (0:3) -- (315:3);
    \draw (0:3) -- (270:3);
    \draw (0:3) -- (180:3);
    \draw (0:3) -- (135:3);
    
    \draw (225:3) -- (180:3);
    \draw (225:3) -- (135:3);

    \draw(0:3) circle(10pt);
    \draw(225:3) circle (10pt);
\end{scope}
   \fill(-10,-5.2) node{$a)$};
  \fill(0,-5.2) node{$b)$};
  \fill(10,-5.2) node{$c)$}; 
 \end{tikzpicture}
       \caption{}\label{f1}
  \end{figure}

  
 \item \textbf{Let us choose vertex $3$} and join it with $5,7,8$. (see Fig.\ref{f1} c)). We see that the number of neighbors for $6$ is odd and the number of common neighbors for $6$ and $3$ is odd. We then can choose a vertex with an even number.

\item Let us choose vertex $2$ and join it with $4,5,7,8$ (see Fig.\ref{f2} a)). We see that the number of common neighbors for $2$ and $3$ is even: $N(2) \cap N(3) = \{4,5,7,8\}$ it follows that \underline{we cannot choose vertex $2$ on this step.}
 

\begin{figure}[h!]
    \centering
     \begin{tikzpicture}[scale=0.5]
     \begin{scope}[xshift = -10cm]
             {\foreach \angle/ \label in
       { 90/1, 45/2, 0/3, 315/4, 270/5, 225/6, 180/7, 135/8
        }
     {
        \fill(\angle:3.5) node{$\label$};
        \fill(\angle:3) circle (5pt);
        \draw[gray] (90:3) -- (\angle:3);
      }
    }
    \draw (0:3) -- (315:3);
    \draw (0:3) -- (270:3);
    \draw (0:3) -- (180:3);
    \draw (0:3) -- (135:3);
    
    \draw (225:3) -- (180:3);
    \draw (225:3) -- (135:3);

    \draw[dashed] (45:3) -- (315:3);
    \draw[dashed] (45:3) -- (270:3);
    \draw[dashed] (45:3) -- (180:3);
    \draw[dashed] (45:3) -- (135:3);
    
    \draw(0:3) circle(10pt);
    \draw(225:3) circle (10pt);
    \draw(45:3) circle (10pt);
     \end{scope}
\begin{scope}
    {\foreach \angle/ \label in
       { 90/1, 45/2, 0/3, 315/4, 270/5, 225/6, 180/7, 135/8
        }
     {
        \fill(\angle:3.5) node{$\label$};
        \fill(\angle:3) circle (5pt) ;
     \draw[gray] (90:3) -- (\angle:3);
      }
    }
     \draw (0:3) -- (315:3);
    \draw (0:3) -- (270:3);
    \draw (0:3) -- (180:3);
    \draw (0:3) -- (135:3);
    
    \draw (225:3) -- (180:3);
    \draw (225:3) -- (135:3);

    \draw (315:3) -- (270:3);
    \draw (315:3) -- (180:3);
    \draw (315:3) -- (135:3);
    
    \draw(0:3) circle(10pt);
    \draw(225:3) circle (10pt);
     \draw(315:3) circle (10pt);
\end{scope}
\begin{scope}[xshift = 10cm]
 {\foreach \angle/ \label in
       { 90/1, 45/2, 0/3, 315/4, 270/5, 225/6, 180/7, 135/8
        }
     {
        \fill(\angle:3.5) node{$\label$};
        \fill(\angle:3) circle (5pt) ;
     \draw[gray] (90:3) -- (\angle:3);
      }
    }
     \draw (0:3) -- (315:3);
    \draw (0:3) -- (270:3);
    \draw (0:3) -- (180:3);
    \draw (0:3) -- (135:3);
    
    \draw (225:3) -- (180:3);
    \draw (225:3) -- (135:3);

    \draw (315:3) -- (270:3);
    \draw (315:3) -- (180:3);
    \draw (315:3) -- (135:3);

    \draw[dashed](180:3) --(135:3);
    
    \draw(0:3) circle(10pt);
    \draw(225:3) circle (10pt);
    \draw(315:3) circle (10pt);
    \draw(180:3) circle (10pt); 
\end{scope}
   \fill(-10,-5.2) node{$a)$};
  \fill(0,-5.2) node{$b)$};
  \fill(10,-5.2) node{$c)$}; 
     \end{tikzpicture}
     \caption{}\label{f2}
\end{figure}


\item  \textbf{Let us choose vertex $4$} and join it with vertices $5,7,8$ (see Fig.\ref{f2} b)). It is easy to see that all required conditions hold. Hence we can choose a vertex with an odd number. Let us choose vertex $7$ and join it with $8$ (see Fig.\ref{f2} c)). We see that $N(7) \cap N(4) = \{1,3,8\}$, \textit{i.e.,} $|N(7) \cap N(4)|$ is odd therefore \underline{we cannot chose vertex $7$ on this step.} Thus we have only one possibility.

\item \textbf{Let us choose vertex $5$} and join it with $7,8$ (see Fig.\ref{f3} a)). All conditions hold and hence we have to choose a vertex with an even number. If we chose vertex $8$ (see Fig.\ref{f3} b)) we then get $N(8) \cap N(6) = \{1\}$ \textit{i.e.,} the number of common neighbors for $8$ and $6$ is odd. Hence \underline{we cannot choose vertex $8$ on this step.} It follows that we have only one possibility --- to choose vertex $2$ then $7$ and then $8$.
 

\begin{figure}[h!]
    \centering
     \begin{tikzpicture}[scale=0.5]
     \begin{scope}[xshift = -5cm]
   {\foreach \angle/ \label in
       { 90/1, 45/2, 0/3, 315/4, 270/5, 225/6, 180/7, 135/8
        }
     {
        \fill(\angle:3.5) node{$\label$};
        \fill(\angle:3) circle (5pt) ;
     \draw[gray] (90:3) -- (\angle:3);
      }
    }
     \draw (0:3) -- (315:3);
    \draw (0:3) -- (270:3);
    \draw (0:3) -- (180:3);
    \draw (0:3) -- (135:3);
    
    \draw (225:3) -- (180:3);
    \draw (225:3) -- (135:3);

    \draw (315:3) -- (270:3);
    \draw (315:3) -- (180:3);
    \draw (315:3) -- (135:3);

    \draw (270:3) -- (180:3);
    \draw (270:3) -- (135:3);
    
    \draw(0:3) circle(10pt);
    \draw(225:3) circle (10pt);
    \draw(315:3) circle (10pt);    
    \draw(270:3) circle (10pt);
     
     \end{scope}
\begin{scope}[xshift = 5cm]
   {\foreach \angle/ \label in
       { 90/1, 45/2, 0/3, 315/4, 270/5, 225/6, 180/7, 135/8
        }
     {
        \fill(\angle:3.5) node{$\label$};
        \fill(\angle:3) circle (5pt) ;
     \draw[gray] (90:3) -- (\angle:3);
      }
    }
     \draw (0:3) -- (315:3);
    \draw (0:3) -- (270:3);
    \draw (0:3) -- (180:3);
    \draw (0:3) -- (135:3);
    
    \draw (225:3) -- (180:3);
    \draw (225:3) -- (135:3);

    \draw (315:3) -- (270:3);
    \draw (315:3) -- (180:3);
    \draw (315:3) -- (135:3);

    \draw (270:3) -- (180:3);
    \draw (270:3) -- (135:3);
    
    \draw(0:3) circle(10pt);
    \draw(225:3) circle (10pt);
    \draw(315:3) circle (10pt);    
    \draw(270:3) circle (10pt);
    \draw(135:3) circle (10pt);

\end{scope}

   \fill(-5,-5.2) node{$a)$};
  \fill(5,-5.2) node{$b)$};
     \end{tikzpicture}
     \caption{}\label{f3}
\end{figure}


\item Thus we get the following sequence $S= (1, 6, 3,4,5,2,7,8)$ and the corresponding meander has the following form as it shown in Fig.\ref{meander6}.
  \end{enumerate}

\end{example}


 \begin{figure}[h!]
    \centering
     \begin{tikzpicture}[scale = 2.5]
     \fill[gray!10] (1.5, 0) circle (1.5);
     \draw[thick] (0, 0) to (3, 0);
\draw[ultra thick] (0.2174, 0.777778) to[out = 0, in = 90, distance = 4.18879] (0.333333, 0)
 to[out = -90, in = -90, distance = 20.9439] (2, 0)
 to[out = 90, in = 90, distance = 12.5664] (1, 0)
 to[out = -90, in = -90, distance = 4.18879] (1.33333, 0)
 to[out = 90, in = 90, distance = 4.18879] (1.66667, 0)
 to[out = -90, in = -90, distance = 12.5664] (0.666667, 0)
 to[out = 90, in = 90, distance = 20.9439] (2.33333, 0)
 to[out = -90, in = -90, distance = 4.18879] (2.66667, 0)
to[out = 90, in = 180, distance = 4.18879] (2.7826, 0.777778);
\draw[help lines] (1.5, 0) circle (1.5);
\draw[fill] (0.2174, 0.777778) circle (0.05);
\draw[fill] (2.7826, 0.777778) circle (0.05);
\draw[fill] (0, 0) circle (0.05);
\draw[fill] (3, 0) circle (0.05);

\fill (0, 0) circle (0.05);
\node at (0, 0) [below left] {$p_0$};

\fill (3, 0) circle (0.05);
\node at (3, 0) [below right] {$p_1$};

\fill (0.2174, 0.777778) circle (0.05);
\node at (0.2174, 0.777778) [left] {$m_0$};

\fill (0.333333, 0) circle (0.05);
\node at (0.333333, 0) [below left] {$1$};

\fill (0.666667, 0)  circle (0.05);
\node at (0.666667, 0) [below left] {$2$};

\fill (1, 0) circle (0.05);
\node at (1, 0) [below left] {$3$};

\fill (1.333333, 0) circle (0.05);
\node at (1.333333, 0) [below right] {$4$};

\fill (1.66667, 0) circle (0.05);
\node at (1.66667, 0) [below right] {$5$};

\fill (2, 0) circle (0.05);
\node at (2, 0) [below right] {$6$};

\fill (2.33333, 0) circle (0.05);
\node at (2.33333, 0) [below] {$7$};

\fill (2.66667, 0) circle (0.05);
\node at (2.66667, 0) [below right] {$8$};

\fill (2.7826, 0.777778) circle (0.05);
\node at (2.7826, 0.777778) [below right] {$m_1$};

\end{tikzpicture}
    \caption{The meander with permutation $\pi = \begin{pmatrix}
        1 & 2 & 4 & 4& 5 &6 &7 &8\\
        1 & 6 & 3& 4 & 5 & 2 & 7 & 8
    \end{pmatrix}$}
    \label{meander6}
\end{figure}



\newpage

\begin{thebibliography}{99}
\bibitem[A88]{A88} V.I. Arnol'd, The branched covering $\mathbb{C}P^2 \to S^4$, hyperbolicity and projective topology, {\it Sibirsk. Mat. Zh}, {\bf 29}(5), 237, 36--47, (1988).

\bibitem[B22]{B22} Y. Belousov, Irreducible Meanders, arXiv:2112.102893v3

\bibitem[ECHLPTh]{EpThur} D.B.A. Epstein, I.W. Cannon, D.E. Holt, S.V.F. Levy, M.S. Paterson and W.P. Thurston, {\it Word Processing in Groups, Jones and Bartlett Publishers,} INC., (1999)

\bibitem[DiF95]{DiF95} P. Di Francesco, Folding and coloring problem in mathematical physics, {\it Bull.Amer. Soc}, {\bf 37}(3), 251--307, (2000).

\bibitem[DFGG97]{DFGG97} P. Di Francesco, O. Golinelli, and E. Guitter, Meander, folding, and arch statistics, {\it Math. Comput. Modelling,} {\bf{26}} (8--10), 97-–147, (1997), hep-th/9506030. Combinatorics and physics (Marseilles, 1995). MR1492504

\bibitem[GL20]{GL}  A. Grinblat and V. Lopatkin, On realizability of Gauss diagrams and constructions of meanders, {\it Journal of Knot Theory and Its Ramifications}, {\bf29}(05):2050031, (2020).

\bibitem[KLV22]{KLV} A. Khan,  A. Lisitsa and A. Vernitski, Training AI to Recognize Realizable Gauss Diagrams: The Same Instances Confound AI and Human Mathematicians,Proceedings of the 14th International Conference on Agents and Artificial Intelligence, 990--995, 2022,SCITEPRESS-Science and Technology Publications.

\bibitem[LLV]{LLV} A. Lisitsa, V. Lopatkin, and A. Vernitski, Describing realizable Gauss diagrams using the concepts of parity or bipartite graphs, {\it Journal of Knot Theory and Its Ramifications}, {\bf 32}(10), 2350059, (2023).

\bibitem[M84]{M} G. Moran, Chords in a circle and linear algebra over $\mathsf{GF}(2)$, {\it J. Combin. Theory Ser. A}, {\bf 37}(3),  239--247, (1984).

\bibitem[R83]{R83} P. Rosenstiehl, Planar permutations defined by two intersecting Jordan curves, Graph theory and combinatorics (Cambridge, 1983), Academic Press, London, 259--271, (1984).

\bibitem[STZ09]{STZ09} B. Shtylla, L. Traldi, and L. Zulli, On the realization of double occurrence words, {\it Discrt. Math.,} {\bf 309}(6), 1769--1773, (2009).

\bibitem[Z21]{Z21} A. Zvonkin, Meanders: A personal perspective to the memory of Pierre Rosenstiehl, {\it European Journal of Combinatorics}, In Press, Corrected Proof (2023).

\end{thebibliography}
\end{document}